\documentclass[12pt]{article}
\usepackage{graphicx} 
\usepackage{tikz}
\usepackage[
  pdftex,% For direct pdf compilation compatibility
  plainpages=false,% Usage of not only arabic numbers
  pdfpagelabels,% Enable page labels
  colorlinks,% Colored links
  citecolor=magenta,% Citation color
  linkcolor=blue,% Link color
  urlcolor=black,% URL color
  filecolor=black,% File color
  backref=page,
  colorlinks=true,
  bookmarksopen% For starting document with bookmarks tree opened
]{hyperref} %For pdf links

\usepackage{amssymb, amsfonts, amsmath, amsthm, enumerate}
\usepackage[margin=1in]{geometry}
\usepackage{tikz,etoolbox}
\usepackage{mathtools}
\usepackage[toc,page]{appendix} 
\usepackage{bbm}
\usepackage{bm}
\usetikzlibrary{matrix,calc,shapes,decorations}
\pgfdeclarelayer{background}
\pgfdeclarelayer{main}
\pgfdeclarelayer{foreground}
\pgfsetlayers{background,main,foreground}

\usepackage{graphicx}
\usepackage{float}
\usetikzlibrary{calc}
\usepackage{xcolor}
\newcommand{\icol}[1]
{                 % inline column vector
\left(\begin{smallmatrix}#1\end{smallmatrix}\right)
}    
\newcommand{\C}{\mathbb{C}}

\newcommand{\Z}{\mathbb{Z}}
\newcommand{\R}{\mathbb{R}}
\newcommand{\Q}{\mathbb{Q}}
\renewcommand{\phi}{\varphi}

\DeclareMathOperator*{\lcm}{lcm}
\DeclareMathOperator*{\conv}{conv}

\renewcommand{\P}{{\mathcal P}}

\newcommand{\va}{{\mathbf a}}

\newcommand{\vol}{\operatorname{vol}}

\newcommand{\K}{\mathcal{K}}

\newtheorem{observation}{Observation}
\newtheorem{cor}{Corollary}
\newtheorem{proposition}{Proposition}
\newtheorem{lemma}{Lemma}
\newtheorem{conjecture}{Conjecture}

\newtheorem{eg}{Example}
\newenvironment{example}{\begin{eg}\rm}{\end{eg}}
\newtheorem{problem}{Problem}

\newtheorem{definition}{Definition}
\newtheorem{remark}{Remark}
\usepackage{thmtools}
\usepackage{thm-restate}
\usepackage{cleveref}
\declaretheorem[name=Theorem]{thm}

\newcommand{\smallpolytope}{%
  \mathord{\vcenter{\hbox{%
  \begin{tikzpicture}[scale=0.3, line width=0.4pt]
    \draw (0,0) -- (1,0) -- (1,1) -- (0,1) -- cycle;
    \draw (0.35,0.35) -- (1.35,0.35) -- (1.35,1.35) -- (0.35,1.35) -- cycle;
    \draw (0,0) -- (0.35,0.35);
    \draw (1,0) -- (1.35,0.35);
    \draw (1,1) -- (1.35,1.35);
    \draw (0,1) -- (0.35,1.35);
  \end{tikzpicture}}}}%
}
\title{Half-open integer parallelepipeds and polytope Dedekind sums}

\author{
Sinai Robins\thanks{Gratefully acknowledges the support given by the São Paulo Research Foundation (FAPESP), Grant No.~2023/03167-5.} \and
André Rosenbaum Coelho\thanks{Corresponding author. He gratefully acknowledges the support given by the São Paulo Research Foundation (FAPESP), Grant No.~2025/06117-4.}
}

\date{July 17, 2026}
\begin{document}
\maketitle

\begin{abstract}
We study the Ehrhart theory of half-open $d$-dimensional integer parallelepipeds $\Pi$.  Although the lattice-point count 
$t\Pi\cap \Z^d$ is known to be simply 
\(\vol \Pi t^d\) for positive integer \(t\), the corresponding counting function for arbitrary real dilations $t$ has subtle, nontrivial periodic structure. We give explicit formulas for this real Ehrhart quasi-polynomial, and more generally for all the discrete moments of the real dilates of $\Pi$: $\sum_{p\in t\Pi\cap\mathbb Z^d}\langle p,z\rangle^m$.
The formulas are expressed in terms of Barnes polynomials and polytope Dedekind sums, which encode the periodic lattice flow of translated integer lattices on the flat torus determined by \(\Pi\). Our approach  develops further the study of polytope Dedekind sums, introduced recently in \cite{Robins2026}. In particular, we obtain novel identities for polytope Dedekind sums by using iterated discrete derivatives. Moreover, we show that the Ehrhart quasi-coefficients of $L_\Pi(t)$ are precisely alternating sums of polytope Dedekind sums.  Finally, we give an Ehrhart-type reciprocity law relating \(L_{\Pi}(t)\) at negative arguments to the lattice-point count of the `opposite' half-open parallelepiped.
\end{abstract}

\bigskip
% \noindent\textbf{Acknowledgments.}
% The authors gratefully acknowledge the financial support provided by the São Paulo Research Foundation (FAPESP). André Rosenbaum Coelho was supported by FAPESP Grant No.~2025/06117-4, and Sinai Robins was supported by FAPESP Grant No.~2023/03167-5.

\newpage
\tableofcontents

\medskip

%%%%%%%%%%%%%%%%

%%%%%%%%%%%%%%%%
\section{Introduction and statements of the main results}  \label{sec: intro}

The polytope Dedekind sums were recently introduced in \cite{Robins2026}, where it was shown that they are the workhorse of the Ehrhart quasi-polynomials for all rational polytopes. 

Here we develop further these polytope Dedekind sums. 
Our main geometric object of study is the half-open integer parallelepiped, which we now define. 
Let $\bm w_1,\dots,\bm w_d\in\Z^d$ be $d$ linearly independent primitive integer vectors and $W$ be the $d\times d$ integer matrix whose columns are
$\bm w_1,\ldots,\bm w_d$. 
Since these vectors are linearly independent, $\det W\neq 0$.
The\textbf{ half-open parallelepiped} $\Pi$ may be written in $W$-coordinates as $\Pi=W[0,1)^d$. Equivalently,
\begin{equation}
\label{eq: half-open parallelepiped}
    \Pi=\Big\{\sum_{i=1}^d\lambda_i\bm{w}_i:\lambda_i\in[0,1)\Big\}.
\end{equation}

It is well-known that for positive 
{\bf integer} values of $t$, the Ehrhart quasi-polynomial $L_{\Pi}(t)\coloneq |t\Pi\cap\Z^d|$ is given by
\begin{equation}
    L_\Pi(t) = \vol \Pi \, t^d,
\end{equation}
due to the fact that $\Pi$ tiles $\R^d$ by translations with the generating vectors $\bm w_1,\dots,\bm w_d$. More generally, 
we call a function $f: \R \longrightarrow \C$  a  \textbf{real quasi-polynomial} if 
\begin{equation}
    f(t) = \sum_{k=0}^d c_k(t)t^k,
\end{equation}
where the coefficients $c_k(t)$ are periodic functions of the \textbf{real-valued} parameter $t$.  
% It is well known that the denominator $d(\mathcal P)$ of a rational polytope, i.e., the smallest integer $n$ such that $n\mathcal P$ is integral, is a quasi-period for the classical Ehrhart quasi-polynomial $L_{\mathcal P}(t)$, with $t\in\Z_{\ge0}$. However, there may exist $n<d(\mathcal P)$ such that $n$ is a common period of the quasi-coefficients $c_r(t)$. This phenomena is traditionally called \textbf{period collapse} in the literature. Here, we extend the period investigation for real Ehrhart quasi-polynomials of half-open integer parallelepipeds. 
In the present paper, we address the following problems:

\begin{problem}
\label{endless problem of relations for polytope Dedekind sums}
Give relations between various polytope Dedekind sums
$\mu_k(\Pi,\bm z,t\bm v)$. 
\end{problem}

\begin{problem}
\label{question: all coeff's of the quasi-polynomial for a half-open parallelepiped}
Given any
    {\bf positive real} value of $t$, find an explicit description for all  the coefficients of the real quasi-polynomial 
    $L_\Pi(t)$.
\end{problem}

\begin{problem}
\label{general question: all coeff's of the quasi-polynomial for the m'th moment}
More generally, for 
    {\bf positive real} values of $t$, and any positive integer $m$, give an explicit description for the $m$th discrete moment:
\begin{equation}
    \sum_{\bm p \,\in \,
t \, \Pi \,\cap \,\Z^d} 
\langle \bm p, \bm z \rangle^m.
\end{equation}
\end{problem}

\noindent
In some sense the open direction of Problem \ref{endless problem of relations for polytope Dedekind sums} will continue to see new solutions in the future.  We give one answer to this problem, in Theorem \ref{thm:partial-alternating-differences}.
Problem \ref{question: all coeff's of the quasi-polynomial for a half-open parallelepiped}  has a surprisingly ``simple'' and explicit answer, in terms of polytope Dedekind sums. 
We solve the more general Problem \ref{general question: all coeff's of the quasi-polynomial for the m'th moment} positively, in Theorem 
\ref{thm: moments of dilated half-open parallelepipeds},
part \eqref{thm1:part a} below.

The paper is organized as follows. In Section \ref{sec: intro}, we introduce half-open integer parallelepipeds, their real Ehrhart quasi-polynomials, the associated polytope Dedekind sums, and the Barnes polynomials that appear in our formulas. We then state the main results. Section \ref{sec: dimension two} works out a two-dimensional example in detail, illustrating how the general formula produces an explicit real Ehrhart quasi-polynomial. 
In Section \ref{sec: Ehrhart quasipoly via Barnes and pDs}, we describe the moments of a half-open integer parallelepiped, state the second formula for its Ehrhart quasi-polynomial in terms of Barnes polynomials and polytope Dedekind sums and show how this formula implies the more elementary one, given in Theorem \ref{thm: Ehrhart quasi-polynomial in terms of representatives of the finite abelian group}.
Interestingly, Theorem~\ref{thm:partial-alternating-differences} and Proposition \ref{prop:finite-difference-Bernoulli-Barnes} enjoy an interpretation in the language of iterated discrete derivatives, which we give in Section~\ref{sec: finite difference operators}.  Sections \ref{sec: proof of lemmas} to \ref{sec: proof of theorem moments} contain the proofs of the preliminary lemmas and main theorems. Finally, in  Section \ref{sec: further remarks} we discuss some open problems concerning the periodicity and cancellation phenomena that arise from the lattice-flow interpretation of Ehrhart quasipolynomials.

%%%%%%%%%%%%%%%
\subsection{Our setup}

Our main analytic object of study is the following finite sum, which we will later show can be written as a certain sum over a finite abelian group. 
\begin{definition}
\label{discrete moments}
The {\bf polytope Dedekind sum}
of a half-open integer parallelepiped $\Pi \subset \R^d$ is defined by:
\begin{equation}
\label{def:polytope Dedekind sum}
\mu_k(\Pi, \bm z, t\bm v):= 
\sum_{\bm p \in  
   \Z^d - t\bm v \; (  \hspace{-.1cm}   \bmod \Pi) 
   }    
\langle \bm p, \bm z \rangle^k,
\end{equation}
for each fixed integer $k\geq 0$, 
$\bm v \in \Q^d$, $t\in \R$ and $\bm z\in\C^d$.
\hfill 
\scalebox{.8}{$\smallpolytope$}
\end{definition}

\noindent It is also fruitful to regard
\eqref{def:polytope Dedekind sum}
as the {\bf $k$th discrete moment} of $\Pi$.
From this point of view, the sum is not merely an arithmetic object,
but also a measure of how the lattice points of $\Pi$ are distributed in space;
in this sense it belongs naturally to the language of discrete spatial
statistics.
 
 We begin by recalling some of the setup from \cite{Robins2026}. Fix a direction
\(\bm v \in \mathbb R^d\). For each \(t \in \mathbb R\), consider the translated
integer lattice \(\mathbb Z^d - t\bm v\). 
Reducing this translated lattice modulo
the parallelepiped \(\Pi\), we obtain the lattice flow on the flat torus
determined by \(\Pi\):
\begin{equation}
\operatorname{LatticeFlow}(t)
:=
\mathbb Z^d - t\bm v
\pmod{\Pi}.
\label{eq:lattice-flow}
\end{equation}
When the relevant vertices are rational points, \(\bm v \in \mathbb Q^d\), the
flow is periodic. Thus, as \(t\) ranges over sufficiently large values, the
resulting motion decomposes topologically into a finite collection of closed
geodesics on the flat torus. Each such geodesic is the orbit of a single
integer point of \(\Pi\).  Therefore the polytope Dedekind sums \eqref{def:polytope Dedekind sum} are periodic functions of $t>0$. 

Indeed, because $\bm v=\left(\frac{a_1}{b_1}, \dots, \frac{a_d}{b_d} \right)$ is a rational point, we note that
the period of the polytope Dedekind sum is at most $\lcm(b_1, \dots, b_d)$. 
In particular, when $\bm v \in \Z^d$, and $t$ is also an integer, then 
\[
\mu_k(\Pi, \bm z, t\bm v):=
\mu_k(\Pi, \bm z, 0). 
\]
When $t=0$, we will often write  $\mu_k(\Pi, \bm z, 0):=
\mu_k(\Pi, \bm z)$, a polytope Dedekind sum that no longer has any periodic behavior in $t$.
The {\bf classical Dedekind sum} is defined for any two coprime integers $a, b$ by
\begin{equation}
    s(a,b):=\sum_{k=1}^{b-1}
\left(
\left\{
\frac{ak}{b} 
\right\}-\frac{1}{2}
\right)
\left(
\left\{
\frac{k}{b} 
\right\}-\frac{1}{2}
\right).
\end{equation}

This function arises naturally in discrete geometry, modular forms, the variance of random number generators, signatures of 4-manifolds, special values of L-functions, permutation statistics, and many more areas of mathematics.
They possess a reciprocity law which allows us to compute them in linear-time complexity as a function of their bit input, which is $\log a + \log b$.
%%%%%%%
The polytope Dedekind sums directly extend the classical Dedekind sums and their reciprocity law, as shown in \cite[Appendix~A]{Robins2026}.

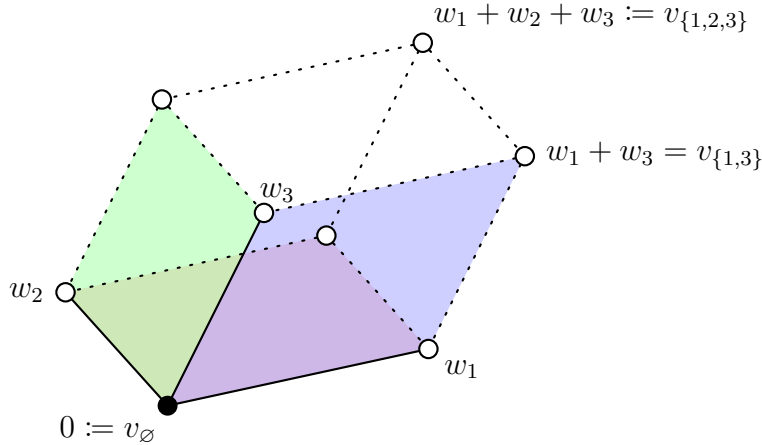
\begin{figure}[H]
\begin{center}
\begin{tikzpicture}[
    scale=1.5,
    line join=round,
    line cap=round
]

%==================================================
% Vertices
%==================================================
\coordinate (O)  at (0,0);

% Outer vectors
\coordinate (W1) at (2.3, 0.5);
\coordinate (W2) at (-.9, 1);
\coordinate (W3) at (0.85, 1.7);

% Remaining vertices
\coordinate (W12)  at ($(W1)+(W2)$);
\coordinate (W13)  at ($(W1)+(W3)$);
\coordinate (W23)  at ($(W2)+(W3)$);
\coordinate (W123) at ($(W1)+(W2)+(W3)$);

%==================================================
% Colored facets
%==================================================

% <w1,w2>
\fill[red!35!white,opacity=.55]
    (O) -- (W1) -- (W12) -- (W2) -- cycle;

% <w1,w3>
\fill[blue!35!white,opacity=.55]
    (O) -- (W1) -- (W13) -- (W3) -- cycle;

% <w2,w3>
\fill[green!35!white,opacity=.55]
    (O) -- (W2) -- (W23) -- (W3) -- cycle;

%==================================================
% Solid edges from the origin
%==================================================
\draw[thick] (O) -- (W1);
\draw[thick] (O) -- (W2);
\draw[thick] (O) -- (W3);

%==================================================
% Remaining box edges dotted
%==================================================
\draw[thick, loosely dotted] (W1) -- (W12);
\draw[thick,loosely dotted] (W2) -- (W12);

\draw[thick,loosely dotted] (W1) -- (W13);
\draw[thick,loosely dotted] (W3) -- (W13);

\draw[thick,loosely dotted] (W2) -- (W23);
\draw[thick,loosely dotted] (W3) -- (W23);

\draw[thick, loosely dotted] (W12) -- (W123);
\draw[thick,loosely dotted] (W13) -- (W123);
\draw[thick,loosely dotted] (W23) -- (W123);

%==================================================
% Vertex spheres
%==================================================

% Black sphere at the origin
\fill[black] (O) circle (2.5pt);

% Empty spheres at the other seven vertices
\foreach \P in {W1,W2,W3,W12,W13,W23,W123}
{
    \fill[white] (\P) circle (2.3pt);
    \draw[thick] (\P) circle (2.3pt);
}

%==================================================
% Labels
%==================================================
\node[below left]  at (O) {$0\coloneq v_{\varnothing}$};

\node[below right] at (W1) {$\, w_1$};
\node[left]        at (W2) {$w_2 \ $};
\node[above right] at (W3) 
{$ \hspace{-2mm}w_3$};

\node[right] at (W13)
{$\ w_1+w_3 = v_{\{1,3\}}$};

\node[above right] at (W123)
{$w_1+w_2+w_3\coloneq v_{\{1,2,3\}}$};
\end{tikzpicture}
\caption{The half-open parallelepiped $\Pi$}
\label{fig: half-open box}
\end{center}
\end{figure}

Each vertex of $\Pi$ may be written as 
\begin{equation}
\bm v_I 
:=
\displaystyle\sum_{i\in I}\bm w_i,
\end{equation}
for some subset $I \subseteq [d]\coloneq  \{1, 2, 3, \dots, d\}$, as suggested by Figure \ref{fig: half-open box}.
An important and recurring theme throughout the paper is the finite abelian group 
\begin{equation}
    G=W^{-1}\mathbb Z^d/\mathbb Z^d,
\end{equation}
where $W$ is the matrix whose columns are $\bm w_1, \dots, \bm w_d$. 
$G$ has the following useful set of coset representatives: 
\begin{equation}
\label{eq: R as the set of representatives of the quocient group}
\mathcal R:=W^{-1}\mathbb Z^d\cap[0,1)^d.
\end{equation}
These are rational points inside the unit cube $[0,1)^d$. The natural isomorphism
$  W^{-1}\Z^d/\Z^d \cong \Z^d/W\Z^d$ given by $ \bm x+\Z^d \mapsto   W\bm x+W\Z^d$
implies that 
$|\mathcal R|
= [\Z^d:W\Z^d] = |\det W|$.
Since $\Pi$ is the half-open parallelepiped generated by the columns of $W$, we also have
\begin{equation}
\label{eq:order-G-volume-low-moment-proof}
|\mathcal R|
=
|\det W|
=
\vol\Pi.
\end{equation}

At each vertex $\bm v_I$, there exists a simplicial vertex tangent cone $\mathcal K^{\#}_{\bm v_I}$, which is {\bf partially-open} and generated by the $\bm w_i's$, but with appropriate signs, as follows.  

\begin{definition}
To each vertex  
$\bm v_I \in \Pi$, we associate the following data.
\begin{enumerate}[(a)]
\item 
{\bf The vertex tangent cone} at $\bm v_I$, with apex at the origin:
\begin{equation}
\mathcal K_{\bm v_I}^{\#}
=\left\{-\sum_{i\in I}\lambda_i\bm w_i+\sum_{j\in I^{c}}\lambda_j\bm w_j:\lambda_i>0\text{, if }i\in I,\lambda_j\ge 0\text{, if }j\in I^{c}\right\}.
\end{equation}
\item 
The \textbf{half-open parallelepiped} with a vertex at the origin:  
\begin{equation}
\Pi^{\#}_{\bm v_I}
=
\left\{-\sum_{i\in I}\lambda_i\bm w_i+\sum_{j\in I^{c}}\lambda_j\bm w_j:\lambda_i\in(0,1]\text{, if }i\in I,\lambda_j\in[0,1)\text{, if }j\in I^{c}\right\}.
\label{eq:half-open parallelepiped at vI}
\end{equation}

\hfill 
\scalebox{.8}{$\smallpolytope$}
\end{enumerate}
\end{definition}
 
For any $\bm z\in \C^d$, and $\bm x \in \R^d$, we will use the slightly unusual, complexified inner product 
$\langle \bm x, \bm z \rangle:= \sum_{m=1}^d x_m z_m$.
The {\bf Barnes polynomials} are defined algebraically by the generating function
\begin{equation}
\label{Barnes}
\frac{x^d e^{t x}}
{(e^{a_1 x}-1)\cdots(e^{a_d x}-1)}
:=
\sum_{k\geq 0}
B_k(t,\va)\frac{x^k}{k!},
\end{equation}
where
$\va:= (a_1,\dots,a_d)
\in \mathbb C^d$
is fixed,  $a_k\not=0, k=1, \dots, d$, and where
$t\in\mathbb R_{>0}$. The identity \eqref{Barnes} is understood for all
$x\in\mathbb C$ in the domain of convergence of the series. It follows
directly from \eqref{Barnes} that, for each nonnegative integer $k$, the
function $B_k(t,\va)$ is a polynomial in $t$.

%%%%%%%%%%%%%%%%%%%%%%%%%%%%%%%%%%%
%%%%%%%%%%%%%%%%%%%%%%%%%%%%%%%%%%%

\subsection{Statements of the results}

In order to prove the main theorems, we will use the following lemma.
It allows us to express polytope Dedekind sums in terms of an explicit periodic function of $t$ summed over a finite abelian group. 

Throughout, we use the finite abelian group of rational points $\mathcal R:= W^{-1}\mathbb Z^d\cap[0,1)^d$, where $W$ is the integral matrix whose columns are $\bm w_1, \dots, \bm w_d$.

\begin{lemma}[Polytope Dedekind sums in \(W\)-coordinates]
\label{lem: Polytope Dedekind sums in W-coordinates}
For every \(t\in\mathbb R\) and 
$m \in \Z_{ \geq 0}$, the polytope Dedekind sums enjoy the following formulation:
\begin{equation}
\label{eq: polytope Dedekind sums in terms of fractional part and explicit representatives}
\mu_m(\Pi,\bm z,t\bm v_I)
=
\sum_{(\beta_1, \dots, \beta_d) \in\mathcal R}
\left(
\sum_{j\notin I}
\langle \bm w_j,\bm z\rangle \beta_j
+
\sum_{j\in I}
\langle \bm w_j,\bm z\rangle
\{\beta_j-t\}
\right)^m,
\end{equation}
where $I \subseteq [d],
\bm z \in \C^d$.
\hfill\hyperlink{proof of lemma: lem: Polytope Dedekind sums in W-coordinates}{(Proof)} $\square$
\end{lemma}

\noindent

The complete Ehrhart quasipolynomial of $\Pi$, valid for all positive dilations $t$, has a surprisingly simple form.

\begin{thm}[The complete Ehrhart quasi-polynomial]
\label{thm: Ehrhart quasi-polynomial in terms of representatives of the finite abelian group}
  We have the following formula for the Ehrhart quasi-polynomial of a half-open integer parallelepiped $\Pi$, for all $t>0$:
    \begin{equation}
\label{eq: Ehrhart quasi-poly in terms of group representatives}
L_\Pi(t)
=
\sum_{(\beta_1, \dots, \beta_d) \in \mathcal R}
\prod_{j=1}^d
\left\lceil t-\beta_j\right\rceil.
\end{equation}

\hfill\hyperlink{proof of theorem:  thm: Ehrhart quasi-polynomial in terms of representatives of the finite abelian group}{(Proof)} $\square$
\end{thm}

%%%%%%%%%%%

\begin{thm}[Ehrhart coefficients and an alternating difference operator]
\label{thm:partial-alternating-differences}
 For each subset $S\subseteq[d]$, 
we define
\begin{equation}
p_S(\Pi, t)
:=
\sum_{(\beta_1, \dots, \beta_d) \in \mathcal R}
\prod_{j\in S}
\bigl(\{\beta_j-t\}-\beta_j\bigr),
\label{eq:DS-definition}
\end{equation}
where
$p_{\varnothing}(\Pi;t)
:=
|\mathcal R|
=
\operatorname{vol}(\Pi)$.
We fix \(\bm z\in\mathbb C^d\), and 
let
$a_j:=\langle \bm w_j,\bm z\rangle$ and $\bm v_I=\sum_{i\in I}\bm w_i$.
Then the following statements hold.
\begin{enumerate}[(a)]
\item
\label{part a of Ehrhart coefficients}
We fix \(S\subseteq [d]\), and 
\(t\in\mathbb R\).  Then for  
each integer \(k\) with \(0\leq k<|S|\), we have:
\begin{equation}
\sum_{I\subseteq S}
(-1)^{|S|-|I|}
\mu_k(\Pi,\bm z,t\bm v_I)
=
0.
\label{eq:partial-vanishing}
\end{equation}

\item
\label{part b of Ehrhart coefficients}
At the critical degree \(k=|S|\), we have:
\begin{equation}
\sum_{I\subseteq S}
(-1)^{|S|-|I|}
\mu_{|S|}(\Pi,\bm z,t\bm v_I)
=
|S|!
\left(\prod_{j\in S}a_j\right)
p_S(\Pi, t).
\label{eq:critical-partial-difference}
\end{equation}

\item
\label{part c of Ehrhart coefficients}
Let $L_\Pi(t)=\sum_{q=0}^d c_q(t)t^q$
be the period-one real quasi-polynomial of $\Pi$. 
Then, for each \(0\leq r\leq d\), we have:
\begin{equation}
c_{d-r}(t)
=
\sum_{\substack{S\subseteq[d]\\|S|=r}}
p_S(\Pi, t).
\label{eq:ehrhart-coefficient-DS}
\end{equation}
In particular, if \(a_j\neq0\) for all \(j\), then
\begin{equation}
c_{d-r}(t)
=
\frac{1}{r!}
\sum_{\substack{S\subseteq[d]\\|S|=r}}
\frac{1}{\prod_{j\in S}a_j}
\sum_{I\subseteq S}
(-1)^{r-|I|}
\mu_r(\Pi,\bm z,t\bm v_I).
\label{eq:ehrhart-coefficient-generic-z}
\end{equation}
\end{enumerate}
\hfill
\hyperlink{proof of theorem: thm:alternating sums for low-rank polytope Dedekind sums}{(Proof)}
 $\square$
\end{thm}
\noindent
See Example \ref{example: nonvanishing of alternating sum for k=d} for an application of Theorem \ref{thm:partial-alternating-differences} to the half-open unit cube.

For $s<0$, we define $L_{\Pi}(s)$ to be 
the step-polynomial continuation of the real Ehrhart function $L_{\Pi}(t)$ with $t>0$.  Moreover, we define 
\begin{equation}
\label{eq: THE FLIP of the half-open parallelepiped}
    \Pi^* :=\Big\{\sum_{i=1}^d\lambda_i\bm{w}_i:\lambda_i
    \in (0,1]\Big\},
\end{equation}
and note that
$\Pi^*= -\Pi + \bm w_1 + \cdots + \bm w_d$. 

Stanley proved a general reciprocity law for all partially open polytopes \cite{Stanley1974}. Here we merely extend Stanley's reciprocity law to all positive dilates of half-open parallelepipeds.

\begin{thm}[Ehrhart-type reciprocity law for half-open parallelepipeds]
\label{Ehrhart-type reciprocity law}
 
Let $\Pi  \subset \R^d$ be a half-open integer parallelepiped.  Then for all real $t>0$, we have:
\begin{equation}
L_{\Pi}(-t) = (-1)^d L_{\Pi^*}(t).
\end{equation}

\hfill\hyperlink{proof of theorem:  Ehrhart-type reciprocity law}{(Proof)} $\square$
\end{thm}
%\hfill
%\hyperlink{proof of theorem: thm:alternating sums for low-rank polytope Dedekind sums}{(Proof)}
% $\square$

\begin{remark}\label{remark: nonvanishing of alternating sum for k=d}
Part \eqref{part c of Ehrhart coefficients} of Theorem \ref{thm:partial-alternating-differences} has an equivalent formulation, using the following particular choice for $z$.  For each \(S\subseteq[d]\), let
$
\mathbf 1_S:=\sum_{j\in S}\bm e_j
$, and 
$\bm z_S:=W^{-T}\mathbf 1_S$.
It follows that:
\[
\langle \bm w_j,\bm z_S\rangle
=
\begin{cases}
1,&j\in S,\\
0,&j\notin S.
\end{cases}
\]
Thus \eqref{eq:critical-partial-difference} becomes
\begin{equation}
\label{D_S formula}
p_S(\Pi, t)
=
\frac{1}{|S|!}
\sum_{I\subseteq S}
(-1)^{|S|-|I|}
\mu_{|S|}(\Pi,\bm z_S,t\bm v_I).
\end{equation}
Substituting \eqref{D_S formula} into \eqref{eq:ehrhart-coefficient-DS} gives us:
\begin{equation}
c_{d-r}(t)
=
\frac{1}{r!}
\sum_{\substack{S\subseteq[d]\\|S|=r}}
\sum_{I\subseteq S}
(-1)^{r-|I|}
\mu_r(\Pi,\bm z_S,t\bm v_I).
\label{eq:ehrhart-coefficient-critical-dedekind}
\end{equation}
Finally, Substituting \eqref{D_S formula} into 
\eqref{eq:ehrhart-DS-expansion} gives us:
\begin{equation}
L_\Pi(t)
=
\sum_{S\subseteq[d]}
\frac{t^{d-|S|}}{|S|!}
\sum_{I\subseteq S}
(-1)^{|S|-|I|}
\mu_{|S|}(\Pi,\bm z_S,t\bm v_I).
\label{eq:ehrhart-critical-dedekind-expansion}
\end{equation}
\end{remark}

\begin{remark}[Equivalent forms of the quasi-constant coefficient]
\label{cor:constant-coefficient-forms}
By  \eqref{eq:ehrhart-coefficient-DS}, the quasi-constant coefficient is
\begin{align}
c_0(t)
=p_{[d]}(\Pi,t)=
\sum_{\beta\in \mathcal R}
\prod_{j=1}^{d}
\bigl(\{\beta_j-t\}-\beta_j\bigr).
\label{eq:c0-ceilings}
\end{align}
More generally, fix \(\bm z\in\mathbb C^d\) and put
$a_j:=\langle \bm w_j,\bm z\rangle$.
If \(a_j\neq0\) for every \(j\), then
\begin{align}
c_0(t)
&=
\frac{(-1)^d}{d!\,a_1\cdots a_d}
\sum_{I\subseteq[d]}
(-1)^{|I|}
\mu_d(\Pi,\bm z,t\bm v_I).
\label{eq:c0-polytope-dedekind-theorem-one-sign}
\end{align}
Thus $c_0(t)$ is exactly the first nonvanishing, full $d$-fold alternating difference of the polytope Dedekind sums.
In particular, let
$\mathbf 1:=(1,\ldots,1)^T,
\bm z_{[d]}:=W^{-T}\mathbf 1$.
Since
$\langle \bm w_j,\bm z_{[d]}\rangle=1$ for 
$1\leq j\leq d$, 
we obtain the normalized formula
\begin{align}
c_0(t)
=
\frac{(-1)^d}{d!}
\sum_{I\subseteq[d]}
(-1)^{|I|}
\mu_d(\Pi,\bm z_{[d]},t\bm v_I).
\label{eq:c0-normalized-theorem-one-sign}
\end{align}
\end{remark}

For further background on Ehrhart polynomials and quasi-polynomials, the
reader may consult the books \cite{Barvinok}, \cite{CCD},
\cite{Robinsbook}, \cite{Stanley}, as well as the papers
 \cite{BeckSamWoods}, \cite{Ehrhart},
\cite{Robins2026}.

%%%%%%%%%%%%%%%%%%

%%%%%%%%%%%%%%%%%%

\section{Some examples}
\label{sec: dimension two}
In this section, we compute explicitly the Ehrhart quasi-polynomial of a
specific half-open parallelepiped, using Theorem
\ref{thm: Ehrhart quasi-polynomial in terms of representatives of the finite abelian group}.
We also illustrate Theorem \ref{thm:partial-alternating-differences} by deriving 
a closed form for the alternating sum of the polytope Dedekind sum of the unit cube.

\begin{example}
\label{ex: 2d example}
Let
$\bm w_1=(1,0)$, 
\, $\bm w_2=(1,2)$,
and 
$
W=
\begin{pmatrix}
1&1\\
0&2
\end{pmatrix}$. 
We recall $\Pi :=W\ [0,1)^2$, so that here 
\(
\vol\Pi=|\det W|=2.
\)

The group \(W^{-1}\mathbb Z^2/\mathbb Z^2\) has precisely the two representatives, namely
$(0,0)$ and $\left(\frac12,\frac12\right)$. 
Therefore, by Theorem \ref{thm: Ehrhart quasi-polynomial in terms of representatives of the finite abelian group}, we have
\begin{align*}
    L_{\Pi}(t)&=\lceil t-0\rceil^2+\lceil t-\frac{1}{2}\rceil^2\\
    &=\lceil t\rceil^2+\lceil t-\frac{1}{2}\rceil^2.
\end{align*}
Using the identity $\lceil x\rceil=x+\{-x\}$ we obtain
\begin{equation}
    L_{\Pi}(t)=(t+\{-t\})^2+\left(t-\frac{1}{2}+\left\{\frac{1}{2}-t\right\}\right)^2.
\end{equation}
Expanding, we get
\begin{equation}
    L_{\Pi}(t)=2t^2+2\left(\{-t\}+\left\{\frac{1}{2}-t\right\}-\frac{1}{2}\right)t+\{-t\}^2+\left(\left\{\frac{1}{2}-t\right\}-\frac{1}{2}\right)^2.
\end{equation}
As a sanity check, if \(t=N\) is a positive integer, then  $\{-N\}=0$
and 
$\left\{\frac12-N\right\}=\frac12$.
So we retrieve the known formula 
$L_\Pi(N)=2N^2$.

Clearly $c_1(t)=2\left(\{-t\}+\left\{\frac{1}{2}-t\right\}-\frac{1}{2}\right)$ has fundamental period $\dfrac{1}{2}$. Hence the linear coefficient exhibits period collapse, although the full coefficient system still has common period 1.
\end{example}

%%%%%%%
\begin{example}
\label{example: nonvanishing of alternating sum for k=d}
We consider the half-open unit cube
$\Pi=[0,1)^d$, with 
$\bm w_j:=\bm e_j, 
\bm z:=(1,\ldots,1)$. 
We claim that for every integer $k\geq 0$ and every $t\in\R$:
\begin{equation}
\label{eq:unit-cube-stirling-closed-form}
\sum_{I\subseteq[d]}
(-1)^{|I|}
\mu_k\bigl(\Pi,\bm z,t\bm v_I\bigr)
=
\begin{cases}
(-1)^d d!\,S(k,d)\{-t\}^k,
    & k>d,\\[2mm]
(-1)^d d!\{-t\}^d,
    & k=d,\\[2mm]
0,
    & 0\leq k<d,
\end{cases}
\end{equation}
where $S(k,d)$ denotes the Stirling number of the second kind.
In particular, the alternating sum vanishes below the critical
degree $d$. But we see from 
\eqref{eq:unit-cube-stirling-closed-form} that
at the critical degree, and above, if we suppose that $t\notin\Z$, then:
\begin{equation}
\label{eq:unit-cube-nonzero-for-noninteger-t}
\sum_{I\subseteq[d]}
(-1)^{|I|}
\mu_k\bigl(\Pi,\bm z,t\bm v_I\bigr)
\neq 0,
\qquad k\geq d.
\end{equation}

\noindent
To prove \eqref{eq:unit-cube-stirling-closed-form}, we first note
that the finite group $\mathcal R$ is trivial for the unit cube.
Moreover,
\begin{equation}
\label{eq:unit-cube-translated-lattice-point}
\bigl(\Z^d-t\bm v_I\bigr)\cap[0,1)^d
=
\sum_{i\in I}\{-t\}\bm e_i,
\end{equation}
and since $\bm z=(1,\ldots,1)$, it follows that here the polytope Dedekind sums are:
\begin{equation}
\label{eq:unit-cube-moment}
\mu_k\bigl(\Pi,\bm z,t\bm v_I\bigr)
:=
\sum_{\bm p \in  
   \Z^d - t\bm v \; (  \hspace{-.1cm}   \bmod \Pi) 
   }    
\langle \bm p, \bm z \rangle^k
=
\bigl(|I|\{-t\}\bigr)^k.
\end{equation}
We therefore have:
\begin{align}
\sum_{I\subseteq[d]}
(-1)^{|I|}
\mu_k\bigl(\Pi,\bm z,t\bm v_I\bigr)
&=
\{-t\}^k
\sum_{r=0}^d
(-1)^r
\binom{d}{r}r^k
\label{eq:unit-cube-binomial-sum}\\
&=
(-1)^d d!\,S(k,d)\{-t\}^k,
\label{eq:unit-cube-Stirling-evaluation}
\end{align}
where we used the standard identity
$
\sum_{r=0}^d
(-1)^r
\binom{d}{r}r^k
=
(-1)^d d!\,S(k,d)
$. 
Finally,
\begin{equation}
\label{eq:relevant-Stirling-values}
S(k,d)
=
\begin{cases}
0, & 0\leq k<d,\\
1, & k=d,\\
S(k,d)>0, & k>d.
\end{cases}
\end{equation}
Substituting these three cases into
\eqref{eq:unit-cube-Stirling-evaluation} proves
\eqref{eq:unit-cube-stirling-closed-form}. 
\hfill\scalebox{.8}{$\smallpolytope$}
\end{example}

%%%%%%%%%%%%%%%%%%%%%%%%%%%%%%%%%%
%%%%%%%%%%%%%%%%%%%%%%%%

\section{\texorpdfstring{$L_{\Pi}(t)$}{LPi(t)} via Barnes polynomials and polytope Dedekind sums}
\label{sec: Ehrhart quasipoly via Barnes and pDs}

In this section, we provide an  explicit expression for all the moments of $t\Pi$, with $t>0$, in terms of Barnes polynomials and polytope Dedekind sums, using the ingredients of \cite{Robins2026}. 
While Theorem \ref{thm: Ehrhart quasi-polynomial in terms of representatives of the finite abelian group} gives an elementary formula for the Ehrhart quasi-polynomial of a half-open integer parallelepiped, Theorem \ref{thm: moments of dilated half-open parallelepipeds} below gives an alternative formulation by means of its $0$'th moment. Moreover, we show in this section that 
Theorem \ref{thm: moments of dilated half-open parallelepipeds}, together with an auxiliary argument, furnishes us with another proof of 
Theorem \ref{thm: Ehrhart quasi-polynomial in terms of representatives of the finite abelian group}.

We note that the reason for in this section is not merely the implication of Theorem \ref{thm: Ehrhart quasi-polynomial in terms of representatives of the finite abelian group} from Theorem \ref{thm: moments of dilated half-open parallelepipeds}.  Of more interest are the tools that give us the reduction arguments, used in the proof of Theorem \ref{thm: moments of dilated half-open parallelepipeds}, which we hope might be useful in reductions of Theorem \ref{thm: moments of dilated half-open parallelepipeds} to other families of rational polytopes in the future.  

The starting point for some of our proofs is Brion's fundamental result
\cite{Brion}.  To state it, we first define the  {\bf integer point transform} of any bounded set $A \subset \R^d$ by  
\begin{equation}
\label{def:integer point transform of a bounded set}
\sigma_A(\bm z) := \sum_{\bm q \in  A  \cap \Z^d}  e^{\langle \bm q, \bm z \rangle},
\end{equation}
a finite sum. To state Brion's result, we also need the corresponding integer point transform of a polyhedral cone $\K$, defined by
\begin{equation}
\label{def:integer point transform of a cone}
\sigma_{\K}(\bm z) := \sum_{\bm q \in  \K  \cap \Z^d}  e^{\langle \bm q, \bm z \rangle},
\end{equation}
an infinite series that turns out to be a rational function.  We will continue to use the same notation $\sigma_{\K}(\bm z)$ to denote this rational function, the natural meromorphic continuation of the series \eqref{def:integer point transform of a cone}.

The discrete Brion Theorem
reduces the integer point transform of a rational polytope $\P$ to
the finite sum of the integer point transforms of its vertex tangent cones:
\begin{equation}
\label{discrete Brion theorem}
\sigma_{\P}(\bm z) = \sigma_{\bm v_1+\K_{\bm v_1}}(\bm z)  + \cdots +  \sigma_{\bm v_N+\K_{\bm v_N}}(\bm z),
\end{equation}
where the vertices of $\P$ are $\bm v_1, \dots, \bm v_N$. 
Brion's identity  
\eqref{discrete Brion theorem}
is valid
for all $z$ in the set
\begin{equation} \label{def:almost all z}
     \left\{ \bm z\in \C^d \mid 
    \langle  \bm w(\bm v),\bm z \rangle \notin 2\pi i \Z, \text{ for any edge } \bm w(\bm v) \text{ emanating from the vertex }\bm v \text{ of } \P
    \right\}.
\end{equation}
 In our context, we will apply Brion's theorem for half-open parallelepipeds, and the theorem is valid in the following somewhat simpler set:
 \begin{equation}
    \label{eq: for almost all z half-open}
    \{\bm z\in\C^d \ | \ \langle\bm w_j,\bm z\rangle\neq 2\pi i \Z, \ j=1,\dots,d\}.
 \end{equation}
 Throughout, we will state our theorems with the words ``{\bf for almost all $\bm z\in \C^d$}'' to mean that they hold for all $z$ in the set \eqref{eq: for almost all z half-open} above.

The following lemma shows that the integer point transform for a half-open cone is in fact a rational function, which is very similar to the analogous classical case of a closed cone. 
\begin{lemma}[Integer point transform of a partially-open cone]
\label{lemma: integer point transform of the half-open cone}
We fix $I\subseteq[d]$.  Let
\begin{equation}
    \mathcal{K^{\#}}=\{\lambda_1\bm w_1+\dots+\lambda_d\bm w_d:\lambda_i>0\text{, if }i\in I\text{ and }\lambda_j\ge0\text{, if }j\not\in I\}
\end{equation}
be a partially-open simplicial cone, where the generators $\bm w_1,\dots,\bm w_d\in\Z^d$ are primitive vectors.  Moreover, let
\begin{align*}
\Pi^{\#}:=\{\lambda_1\bm w_1+\dots+\lambda_d\bm w_d:\lambda_i\in(0,1]\text{, if }i\in I\text{ and }\lambda_j\in[0,1)\text{, if }j\not\in I\}
    \end{align*}
be a half-open parallelepiped, with a vertex at the origin. Then, for any $\bm v\in\R^d$, we have:
\begin{equation}
    \label{eq: integer point transform of the half-open cone}
\sigma_{\bm v+\mathcal K^{\#}}(\bm z)=\frac{\sigma_{\bm v+\Pi^{\#}}(\bm z)}{(1-e^{\langle \bm w_1,\bm z\rangle})\dots(1-e^{\langle \bm w_d,\bm z\rangle})},
\end{equation}
for almost all complex vectors $\bm z\in \C^d$, as in equation \eqref{eq: for almost all z half-open}.
\hfill
\hyperlink{proof of lemma: integer point transform of the half-open cone}{(Proof)}
 $\square$
\end{lemma}

\begin{lemma}[Integer point transform of a partially-open vertex tangent cone]
\label{lemma:integer point transform of a partially-open vertex tangent cone}
Let $\Pi$ be as in \eqref{eq: half-open parallelepiped} and $\bm v_I+\mathcal K^{\#}_{\bm{v_I}}$ be the vertex tangent cone at the vertex $\bm v_I$. Then the integer point transform of $\bm v_I+\mathcal K^{\#}_{\bm{v_I}}$ is given by
\begin{equation}
\label{eq: integer point transform of vertex tangent half-open cone}
\sigma_{\bm v_I+\mathcal K^{\#}_{\bm {v_I}}}(\bm z)
=
\frac{\sigma_{\bm {v_I}+\Pi^{\#}_{\bm {v_I}}}(\bm z)}{\prod_{j\in I^{c}}(1-e^{\langle \bm w_j,\bm z\rangle})\prod_{i\in I}(1-e^{-\langle \bm w_i,\bm z\rangle})},
\end{equation}
for almost all complex vectors $\bm z\in \C^d$, as in equation \eqref{eq: for almost all z half-open}.
\hfill
\hyperlink{proof of lemma: integer point transform of a partially-open vertex tangent cone}{(Proof)}
 $\square$
\end{lemma}

\begin{observation}
By Lemma \ref{lemma:integer point transform of a partially-open vertex tangent cone}, any 
vertex tangent cone $\K_v$ has a half-open integer parallelepiped $\Pi \subset \R^d$ that is generated by the primitive edge vectors $\epsilon_1 \bm w_1, \dots \epsilon_d \bm w_d$, for some choices of $\epsilon_j \in \{1, -1\}$.  Consequently, all of the vertex parallelepipeds of $\Pi$  have the same volume: 
\begin{equation}
    \vol \Pi^{\#}_{\bm {v_I}} = \vol \Pi,
\end{equation}
for all $I\subseteq[d]$.
\hfill 
\scalebox{.8}{$\smallpolytope$}
\end{observation}

To set notation, we let
$\va = (\langle \bm w_1,\bm z\rangle, \dots, 
\langle \bm w_d,\bm z\rangle)$,
where the $\bm w_i$'s are the (primitive) edges of $\Pi$.
 An interesting result is an inclusion-exclusion formula for Barnes polynomials, which leads to powerful simplifications, and may be of independent interest.

\begin{proposition}[Inclusion-exclusion formula for Barnes polynomials]
\label{prop:finite-difference-Bernoulli-Barnes}
Fix $\va=(a_1,\ldots,a_d)\in \C^d$, with $a_i\neq0, \ i=1,\dots,d$ and 
 define
$
a_I:=\sum_{i\in I}a_i.
$
Then, for every $t\in\R_{>0}$, we have
\begin{equation}
\frac{(-1)^d}{d!}
\sum_{I\subseteq[d]}
(-1)^{|I|}
B_k(ta_I,\va)
=
\begin{cases}
t^d, & k=d,\\
0, & k< d, \\
0, & k > d, \ t = 1.
\end{cases}
\end{equation}
Here the summation is extended over all $2^d$ subsets 
$I\subseteq [d]$.
\hfill\hyperlink{proof of proposition: prop:finite-difference-Bernoulli-Barnes}{(Proof)} $\square$
\end{proposition}

Theorem 
\ref{thm:partial-alternating-differences}
and Proposition \ref{prop:finite-difference-Bernoulli-Barnes} enjoy an interpretation in the language of iterated discrete derivatives operators, which we give in Section \ref{sec: finite difference operators}.
We also mention that
Proposition \ref{prop:finite-difference-Bernoulli-Barnes} is closely related to the
specialization at $x=0$ of the iterated discrete derivative identity for multiple
Bernoulli polynomials appearing in Appendix A of \cite{KidwaiOsuga2025}.

\noindent

The main result in this section is the following formulation for any discrete moment of a half-open integer parallelepiped.

\begin{thm}[Moments of dilated half-open parallelepipeds]
\label{thm: moments of dilated half-open parallelepipeds}
%\label{thm: moments of dilated half-open parallelepipeds}

Let $\Pi \subset \R^d$ be a half-open integer parallelepiped, fix any real $t>0$ and $m\in\Z_{\ge0}$. Then, we have the following:

\begin{enumerate}[(a)]
\item 
\label{thm1:part a}
The moments of $t\Pi$ are given by:
\begin{align}
\hspace{-1em}
\sum_{\bm p \,\in \,
t \, \Pi \,\cap \,\Z^d} 
\langle \bm p, \bm z \rangle^m
 &=
\tfrac{(-1)^d m! }{(d+m)!}
\sum_{I\subseteq[d]}    
\sum_{k=0}^{d+m}
\tbinom{d+m}{k}  
(-1)^{|I|} 
 B_k\big(t  
 \langle \bm v_I, \bm{z} \rangle,  \va
 \big)
\, \mu_{d+m-k}\left(
\Pi, \bm z, (t-1)\bm v_I
\right),
\end{align}
for almost all complex vectors $\bm z\in \C^d$ as in the set \eqref{eq: for almost all z half-open}. The right-hand side is summed over all $2^d$ subsets of $[d]$, each subset representing a vertex of $\Pi$, the $B_k(x,\va)$'s are Barnes polynomials, defined in \eqref{Barnes} and the $\mu_k(\Pi,\bm z,\bm v)$ are the polytope Dedekind sums, defined in \eqref{def:polytope Dedekind sum}.

\item 
\label{cor:Ehrhart quasi-polynomial of a half-open integer parallelepiped}
In particular, for any $t>0$ the Ehrhart quasi-polynomial of $\Pi$ is:
\begin{align}
\label{Ehrhart poly of half-open parallelepiped}
L_\Pi(t) 
 = \frac{(-1)^d}{d!}
 \sum_{I\subseteq[d]}\sum_{k=0}^{d}  
\binom{d}{k} 
 (-1)^{|I|}
 B_k\big(  t\langle \bm v_I, \bm{z} \rangle,   \va
 \big)
\, \mu_{d-k}\big(
\Pi, \bm z, (t-1)\bm v_I
\big).
 \end{align}
\end{enumerate}

\hfill
\hyperlink{proof of theorem: thm: moments of dilated half-open parallelepipeds}{(Proof)}
 $\square$
\end{thm}

\subsection{Derivation of Theorem \ref{thm: Ehrhart quasi-polynomial in terms of representatives of the finite abelian group} from  Theorem \ref{thm: moments of dilated half-open parallelepipeds}}

We now show that the Barnes polynomial formula in Theorem \ref{thm: moments of dilated half-open parallelepipeds} implies
the elementary group representative formula in Theorem \ref{thm: Ehrhart quasi-polynomial in terms of representatives of the finite abelian group}.
%%%%%%%

\begin{proof}[Proof that Theorem \ref{thm: moments of dilated half-open parallelepipeds} implies Theorem \ref{thm: Ehrhart quasi-polynomial in terms of representatives of the finite abelian group}]
For convenience, put
\begin{equation}
a_j:=\langle \bm w_j,\bm z\rangle,
\qquad
\bm a:=(a_1,\ldots,a_d),
\qquad
a_I:=\sum_{j\in I}a_j
      =\langle \bm v_I,\bm z\rangle
\end{equation}

\noindent
By Theorem \ref{thm: moments of dilated half-open parallelepipeds} \eqref{cor:Ehrhart quasi-polynomial of a half-open integer parallelepiped}, and Lemma \ref{lem: Polytope Dedekind sums in W-coordinates}, we have:
\begin{align}
L_{\Pi}(t)
&=
\frac{(-1)^d}{d!}
\sum_{\bm\beta\in R}
\sum_{I\subseteq[d]}
(-1)^{|I|}
\sum_{k=0}^{d}
\binom{d}{k}
B_k(ta_I,\bm a)\,
X_{I,\bm\beta}^{\,d-k},
\end{align}
where
\begin{equation}
X_{I,\bm\beta}
:=
\sum_{j\notin I}a_j\beta_j
+
\sum_{j\in I}a_j\{\beta_j-t\}.
\end{equation}
The Barnes addition formula \cite[Appendix A]{Robins2026} gives
\begin{equation}
\sum_{k=0}^{d}
\binom{d}{k}
B_k(ta_I,\bm a)\,
X_{I,\bm\beta}^{\,d-k}
=
B_d(ta_I+X_{I,\bm\beta},\bm a).
\end{equation}
We define
$N_j:=\lceil t-\beta_j\rceil$ and 
$x_{\bm\beta}:=\sum_{j=1}^{d}a_j\beta_j$.
Since
\begin{equation}
t+\{\beta_j-t\}
=
\beta_j+\lceil t-\beta_j\rceil
=
\beta_j+N_j,
\end{equation}
we have
\begin{equation}
ta_I+X_{I,\bm\beta}
=
x_{\bm\beta}+\sum_{j\in I}N_ja_j.
\end{equation}
Consequently,
\begin{equation}
\label{Penultimate quasipoly}
L_{\Pi}(t)
=
\frac{(-1)^d}{d!}
\sum_{\bm\beta\in \mathcal R}
\sum_{I\subseteq[d]}
(-1)^{|I|}
B_d\left(
x_{\bm\beta}+\sum_{j\in I}N_ja_j,\bm a
\right).
\end{equation}
\noindent
Now
\begin{equation}
\label{Leading term in Barnes}
B_d(x,\bm a)
=
\frac{x^d}{a_1\cdots a_d}+Q(x),
\qquad
\deg Q<d.
\end{equation}
By precisely the vanishing criterion used in the proof of
Theorem \ref{thm:partial-alternating-differences}, the \(d\)-fold alternating iterated discrete derivative of \(Q\)
vanishes. Hence only the leading term in \eqref{Leading term in Barnes} contributes:
\begin{align}
&\sum_{I\subseteq[d]}
(-1)^{|I|}
B_d\left(
x_{\bm\beta}+\sum_{j\in I}N_ja_j,\bm a
\right)
\notag\\
&\qquad
=
\frac{1}{a_1\cdots a_d}
\sum_{I\subseteq[d]}
(-1)^{|I|}
\left(
x_{\bm\beta}+\sum_{j\in I}N_ja_j
\right)^d
\notag\\
&\qquad
=
(-1)^d d!\prod_{j=1}^{d}N_j.
\end{align}
Substitution into 
\eqref{Penultimate quasipoly} therefore yields
\begin{equation}
L_{\Pi}(t)
=
\sum_{\bm\beta\in \mathcal R}
\prod_{j=1}^{d}N_j
=
\sum_{\bm\beta\in \mathcal R}
\prod_{j=1}^{d}\lceil t-\beta_j\rceil.
\end{equation}
\end{proof}
%%%%%%%%%%%%

%%%%%%%%%%%%%%%%%%%%%%%%%%%
%%%%%%%%%%%%%%%%%%%%%%%%%%%

\section{Interpretation of Theorem 
\ref{thm:partial-alternating-differences} 
as the vanishing of an iterated discrete derivative}
\label{sec: finite difference operators}

In the spirit of Gunnels and Sczech's work \cite{GunnelsSczech} on Dedekind sums as Eisenstein cocycles, 
Theorem \ref{thm:partial-alternating-differences}
may be interpreted 
as a vanishing of an iterated discrete derivative operator.  Additionally, Proposition \ref{prop:finite-difference-Bernoulli-Barnes} may also be interpreted in a similar manner.  Here we detail both interpretations, noting that these interpretations suggest a future relation to coboundary operators.

Difference operators are ubiquitous in discrete geometry and combinatorics.  To link some of our results to these operators, we define
the {\bf discrete derivative in the direction $\bm w \in \R^d$} by:
\begin{equation}
    \Delta_{\bm w} F(\bm x) := 
F(\bm x+ \bm w) -  F(\bm x).
\end{equation}

\noindent
This operator is also known as a forward difference operator. 
Equivalently, if $T_w$ denotes the translation operator
$ (T_{\bm w}F)(\bm x):=F(\bm x+ \bm w)$, then 
$  \Delta_{\bm w} F(\bm x):= (T_{\bm w} -1)F(\bm x)$.

\begin{definition}[Iterated discrete derivative]
\label{def:Iterated discrete derivative}
Let $\Lambda \subset \mathbb{R}^n$ be a full-rank lattice,  and fix
$\bm w_1,\ldots, \bm w_d \in \Lambda$.  Suppose we have a function defined on any translate of $\Lambda$:
\begin{equation}
\label{function on a lattice translate}
        F:\Lambda + \bm x \longrightarrow \C,
\end{equation}
where $\bm x \in \R^d$ is any fixed vector. 
  The {\bf iterated discrete derivative}
 is  defined by
\begin{equation}
\label{iterated discrete derivative}
\left(\Delta_{\bm w_1,\ldots, \bm w_d}F\right)(\bm x)
:=
\sum_{I\subseteq [d]}
(-1)^{d-|I|}
F\!\left(\bm x+\sum_{i\in I}
\bm w_i\right).
\end{equation}
\hfill 
\scalebox{.8}{$\smallpolytope$}
\end{definition}
By inclusion-exclusion, it follows easily that \eqref{iterated discrete derivative} is equivalent to:
\[
\Delta_{\bm w_1,\ldots, \bm w_d} 
=
(T_{\bm w_1}-1)(T_{ \bm w_2}-1)\cdots (T_{\bm w_d}-1).
\]
Thus $\Delta_{\bm w_1,\ldots, \bm w_d}$ is the operator obtained by
taking $d$ successive discrete derivatives, in the directions
$\bm w_1,\ldots, \bm w_d$.
We now apply this language to polytope Dedekind sums.  For fixed
$\Pi$ and $\bm z$, define the function
\begin{equation}
    \Phi_k(\bm x)
:=
\mu_k(\Pi,\bm z,\bm x).
\end{equation}

Hence $\Phi_k$ may be thought of as a function of the translated lattice:
$\Phi_k: 
\Lambda + \bm x \longrightarrow \C$, 
as in \eqref{function on a lattice translate}.
Recalling the definition
$\bm v_I:=\sum_{i\in I}\bm w_i$,
the alternating sum appearing in
Theorem
\ref{thm:partial-alternating-differences}
has the form
\begin{equation}
    \sum_{I\subseteq [d]}
        (-1)^{d-|I|}
        \Phi_k(t\bm v_I).
\end{equation}
        
Therefore the case $S=[d]$ in Theorem
\ref{thm:partial-alternating-differences} part \eqref{part a of Ehrhart coefficients}
may be equivalently interpreted as the following vanishing statement:
\begin{align}
\left(
\Delta_{\bm w_1,\ldots,\bm w_d}\Phi_k
\right)
(t\bm v_I)&:=
\sum_{I\subseteq [d]}
(-1)^{d-|I|}
\mu_k\!\left(\Pi,\bm z,t\bm v_I+\sum_{i\in I}
\bm w_i\right)\\
&=
\sum_{I\subseteq [d]}
(-1)^{d-|I|}
\mu_k\!\left(\Pi,\bm z,t\bm v_I+\bm v_I\right)\\
&=
0,
\end{align}
for $0\le k<d$.
In other words, the rank-$k$ polytope Dedekind sum enjoys a vanishing
$d$-fold iterated discrete derivative in the directions
$\bm w_1,\ldots,\bm w_d$.

%%%%%%%%%%%%%

%%%%%%%%%%%%%%%%%%%%%%%

\section{Proof of Lemma \ref{lem: Polytope Dedekind sums in W-coordinates}}
\label{sec: proof of lemmas}

%%%%%%%%%%%%%%%%%%%%%%%%%%%%%%
\bigskip

\begin{proof}
\hypertarget{proof of lemma: lem: Polytope Dedekind sums in W-coordinates} (of Lemma \ref{lem: Polytope Dedekind sums in W-coordinates})
We first construct a one-to-one correspondence between the points in the half-open parallelepiped $\Pi$ and certain rational points in the half-open unit cube. 
By definition, every point \(\bm p\in\Pi\) has a unique representation
\begin{equation}
\label{p is in Pi}
\bm p=W\bm x=\sum_{j=1}^d x_j\bm w_j,
\qquad
\bm x\in[0,1)^d.
\end{equation}
The condition
\(
\bm p\in \mathbb Z^d-t\bm v_I\pmod{\Pi}
\)
means
\(
\bm p+t\bm v_I\in\mathbb Z^d.
\)
We already know from \eqref{p is in Pi}
that $ p \in \Pi$. 
Since
\(
\bm v_I=W\bm 1_I,
\)
where $\bm 1_I=\sum_{i\in I}\bm e_i$, this is equivalent to
\(
W(\bm x+t\bm 1_I)\in\mathbb Z^d,
\)
or
\(
\bm x+t\bm 1_I\in W^{-1}\mathbb Z^d.
\)
We define 
\[
\mathcal R := W^{-1}\Z^d\cap [0,1)^d.
\]
Thus there is a unique representative
\(
\bm\beta=(\beta_1,\ldots,\beta_d)\in\mathcal R
\)
such that
\[
\bm x+t\bm 1_I\equiv \bm\beta \pmod{\mathbb Z^d}.
\]
Since \(\bm x\in[0,1)^d\), we get
\[
x_j=
\begin{cases}
\beta_j, & j\notin I,\\
\{\beta_j-t\}, & j\in I.
\end{cases}
\]

\noindent
To summarize, we now have a bijection 
$\bm p \rightarrow \bm \beta$.
Therefore
\[
\langle \bm p,\bm z\rangle
=
\sum_{j\notin I}a_j\beta_j
+
\sum_{j\in I}a_j\{\beta_j-t\}.
\]
Raising to the $m$-th power and summing over all representatives
$\bm\beta\in\mathcal R$, we obtain
\begin{equation}
\label{eq:muk-expanded-beta-sum-low-moment-proof}
\mu_m(\Pi,\bm z,t\bm v_I)
=\sum_{\bm p \in  
   \Z^d - t\bm v_I \; (  \hspace{-.1cm}   \bmod \Pi) 
   }    
\langle \bm p,\bm z\rangle^m=
\sum_{\bm\beta\in\mathcal R}
\left(
\sum_{j\notin I}a_j\beta_j
+
\sum_{j\in I}a_j\{\beta_j-t\}
\right)^m.
\end{equation}

\end{proof}

%%%%%%%%%%%%%%

\section{Proof of Theorem \ref{thm: Ehrhart quasi-polynomial in terms of representatives of the finite abelian group}}
\label{sec: proof of theorem explicit ehrhart quasi-poly in terms of representatives}
\begin{proof}
\hypertarget{proof of theorem: thm: Ehrhart quasi-polynomial in terms of representatives of the finite abelian group} (of Theorem \ref{thm: Ehrhart quasi-polynomial in terms of representatives of the finite abelian group})
As before, let
$\Pi:=W[0,1)^d$, $G=
W^{-1}\Z^d/\Z^d.$  
where 
$W$ is the matrix whose columns are 
$\bm w_1,   \cdots  \bm w_d$.
For \(t>0\),
\begin{equation}
\label{eq:tPi-and-tPi-star-real}
t\Pi
=
W[0,t)^d.\notag
\end{equation}

We identify each class in $G$ with its unique representative
$\bm\beta
=
(\beta_1,\ldots,\beta_d)
\in \mathcal R=W^{-1}\Z^d\cap[0,1)^d$.
Then
\begin{equation}
\label{eq:lattice-points-in-W-coordinates}
\bm p\in t\Pi\cap\Z^d
\quad\Longleftrightarrow\quad
W^{-1}\bm p\in W^{-1}\Z^d\cap[0,t)^d.\notag
\end{equation}
Hence, we have the decomposition
\begin{equation}
    W^{-1}\Z^d
        =
        \bigsqcup_{\bm \beta\in\mathcal R}
        (\bm \beta+\mathbb{Z}^d).\notag
\end{equation}
Thus
\[
        L_{\Pi}(t)
        =
        \sum_{\bm \beta\in \mathcal R}
        \#
        \bigl((\bm \beta+\mathbb{Z}^d)\cap [0,t)^d\bigr).
\]
A point of \(\bm \beta+\mathbb{Z}^d\) has the form $\bm\beta+m=(\beta_1+m_1,\ldots,\beta_d+m_d),\ m\in\mathbb{Z}^d.$
It lies in \([0,t)^d\) precisely when $0\leq \beta_j+m_j<t$, for every $j=1,\ldots,d$ .
The coordinates separate, so
\[
        \#
        \bigl((\bm \beta+\mathbb{Z}^d)\cap [0,t)^d\bigr)
        =
        \prod_{j=1}^d
        \#\{m_j\in \mathbb{Z}:0\leq \beta_j+m_j<t\}.
\]
Since \(0\leq \beta_j<1\), the one-dimensional count is
\[
        \#\{m_j\in \mathbb{Z}:0\leq \beta_j+m_j<t\}
        =
        \lceil t-\beta_j\rceil .
\]
Therefore
\[
        L_{\Pi}(t)
        =
        \sum_{\bm \beta\in \mathcal R}
        \prod_{j=1}^d
        \lceil t-\beta_j\rceil ,
\]
as claimed.

\end{proof}

%%%%%%%%%%%%%%%%%%%%%%
%%%%%%%%%%%%%%%%%%%%%%

\section{Proof of Theorem \ref{thm:partial-alternating-differences}}
\label{sec: proof of theorem alternatinf sums}

\begin{proof}
\hypertarget{proof of theorem: thm:alternating sums for low-rank polytope Dedekind sums}
(of Theorem \ref{thm:partial-alternating-differences})
We fix \(S\subseteq[d]\), and 
put \(r:=|S|\), and
\(\bm \beta=(\beta_1,\ldots,\beta_d)\in \mathcal R\). We define
\[
x_{\bm \beta}:=\sum_{j=1}^d a_j\beta_j
\]
and for each \(j\in S\), we define
\[
h_j
:=
a_j\bigl(\{\beta_j-t\}-\beta_j\bigr).
\]
For \(I\subseteq S\), Lemma~\ref{lem: Polytope Dedekind sums in W-coordinates} gives
\begin{align}
\mu_k(\Pi,\bm z,t\bm v_I)
&=
\sum_{\beta\in \mathcal R}
\left(
\sum_{j\notin I}a_j\beta_j
+
\sum_{j\in I}a_j\{\beta_j-t\}
\right)^k
\notag\\
&=
\sum_{\beta\in \mathcal R}
\left(
x_\beta+\sum_{j\in I}h_j
\right)^k.
\label{eq:mu-partial-difference-form}
\end{align}
Therefore
\begin{align}
&\sum_{I\subseteq S}
(-1)^{r-|I|}
\mu_k(\Pi,\bm z,t\bm v_I)
\notag\\
&\qquad=
\sum_{\beta\in \mathcal R}
\sum_{I\subseteq S}
(-1)^{r-|I|}
\left(
x_\beta+\sum_{j\in I}h_j
\right)^k.
\label{eq:partial-difference-sum}
\end{align}
For any polynomial 
\(f\), we will recall the definition of the iterated discrete derivative
$(\Delta_hf)(x):=f(x+h)-f(x)$ given in Section \ref{sec: finite difference operators}.
The inner sum in \eqref{eq:partial-difference-sum} is the
\(r\)-fold iterated discrete derivative
\[
\left(
\prod_{j\in S}\Delta_{h_j}
\right)x^k
\bigg|_{x=x_\beta}.
\]
An \(r\)-fold iterated discrete derivative annihilates every polynomial of
degree less than \(r\). Thus, if \(k<r\), the inner sum vanishes,
proving part \eqref{part a of Ehrhart coefficients}. 

\noindent
When \(k=r\), one has
\[
\left(
\prod_{j\in S}\Delta_{h_j}
\right)x^r
=
r!\prod_{j\in S}h_j.
\]
It follows that
\begin{align}
\sum_{I\subseteq S}
(-1)^{r-|I|}
\mu_r(\Pi,\bm z,t\bm v_I)
&=
r!\sum_{\beta\in \mathcal R}\prod_{j\in S}h_j
\notag\\
&=
r!
\left(\prod_{j\in S}a_j\right)
\sum_{\beta\in \mathcal R}
\prod_{j\in S}
\bigl(\{\beta_j-t\}-\beta_j\bigr),
\end{align}
proving part \eqref{part b of Ehrhart coefficients}.
We now derive the Ehrhart coefficient formula. Since
\(\beta_j\in[0,1)\), the identity
$x+\{-x\}=\lceil x\rceil$
with \(x=t-\beta_j\) gives
$\lceil t-\beta_j\rceil
=
t-\beta_j+\{\beta_j-t\}
=
t+\{\beta_j-t\}-\beta_j$.
Applying Theorem \ref{thm: Ehrhart quasi-polynomial in terms of representatives of the finite abelian group} and expanding the product, we obtain
\begin{align}
L_\Pi(t)
&=
\sum_{\beta\in \mathcal R}
\prod_{j=1}^d
\left(
t+\{\beta_j-t\}-\beta_j
\right)
\notag\\
&=
\sum_{\beta\in \mathcal R}
\sum_{S\subseteq[d]}
t^{d-|S|}
\prod_{j\in S}
\bigl(\{\beta_j-t\}-\beta_j\bigr)
\notag\\
&=
\sum_{S\subseteq[d]}
t^{d-|S|}
p_S(\Pi, t).
\label{eq:ehrhart-DS-expansion}
\end{align}
Since every \(p_S(\Pi, t)\) is periodic of period \(1\),
grouping the terms in \eqref{eq:ehrhart-DS-expansion} according to
\(|S|=r\) gives
\[
c_{d-r}(t)
=
\sum_{\substack{S\subseteq[d]\\|S|=r}}
p_S(\Pi, t),
\]
which proves  part \eqref{part c of Ehrhart coefficients}. Finally, assuming
\(a_j\neq0\) for all \(j\), we solve 
\eqref{eq:critical-partial-difference} for
\(p_S(\Pi, t)\) and substitute into
\eqref{eq:ehrhart-coefficient-DS},
yielding
\eqref{eq:ehrhart-coefficient-generic-z}.
\end{proof}

%%%%%%%%%%%%%%%%%%%%%%%%%%%%%%%%
%%%%%%%%%%%%%%%%%%%%%%%%%%%%%%%%

%%%%%%%%%%%%%%%%

\section{Proof of Theorem \ref{Ehrhart-type reciprocity law}, an Ehrhart-type reciprocity law}
\label{sec: proof of ehrhart reciprocity}
\begin{proof}
\hypertarget{proof of theorem:  Ehrhart-type reciprocity law} (of Theorem \ref{Ehrhart-type reciprocity law})
We prove this reciprocity law using Theorem \ref{thm: Ehrhart quasi-polynomial in terms of representatives of the finite abelian group}.
We define \(L_\Pi(-t)\) by evaluating this same step-polynomial at \(-t\):
\begin{equation}
\label{eq:L-Pi-negative-step-evaluation}
L_\Pi(-t)
=
\sum_{\bm\beta\in \mathcal R}
\prod_{j=1}^d
\left\lceil -t-\beta_j\right\rceil.
\end{equation}

\noindent
Similarly, we have $t\Pi^*=W(0,1]^d$ and
\begin{equation}
\label{eq:L-Pi-star-coset-sum}
L_{\Pi^*}(t)
=
\sum_{\delta\in \mathcal R}
\prod_{j=1}^d
G_{\delta_j}(t),
\end{equation}
where $G_\alpha(t):= 
\#\{n\in\Z:0<n+\alpha\leq t\}$.  We have 
$G_0(t) = \lfloor t\rfloor$, and 
for  \(0<\alpha<1\) we have
$G_\alpha(t)
=
\lfloor t-\alpha\rfloor+1$.
We now compare the one-dimensional factors. For \(\alpha\in[0,1)\), let
\begin{equation}
\label{eq:alpha-star-definition}
\alpha^*
=
\{-\alpha\}
=
\begin{cases}
0, & \alpha=0,\\
1-\alpha, & 0<\alpha<1.
\end{cases}
\end{equation}
Then, for every \(t>0\),
\begin{equation}
\label{eq:one-dimensional-reciprocity-factor}
\left\lceil -t-\alpha\right\rceil
=
-
G_{\alpha^*}(t).
\end{equation}
Indeed, if \(\alpha=0\), this says
$\lceil -t\rceil = -\lfloor t\rfloor$,
while if \(0<\alpha<1\), it says
\begin{align}
\left\lceil -t-\alpha\right\rceil
=
-\left\lfloor t+\alpha\right\rfloor 
=
-\left(\left\lfloor t-(1-\alpha)\right\rfloor+1\right) 
=
-G_{1-\alpha}(t).
\label{eq:one-dimensional-alpha-positive}
\end{align}

Applying \eqref{eq:one-dimensional-reciprocity-factor} coordinatewise and using the fact that the map
$\gamma\longmapsto -\gamma$
is a bijection of the finite group \(G\), we get:

\begin{align}
L_\Pi(-t)
&=
\sum_{\gamma\in G}
\prod_{j=1}^d
\left\lceil -t-\gamma_j\right\rceil \nonumber
=
(-1)^d
\sum_{\gamma\in G}
\prod_{j=1}^d
G_{\{-\gamma_j\}}(t)  \\
&=
(-1)^d\sum_{\delta\in G}
\prod_{j=1}^d
G_{\delta_j}(t)
=
(-1)^dL_{\Pi^*}(t).
\end{align}
for every real \(t>0\).
\end{proof}

%%%%%%%%%%%%%%%%%%%%%%%%%%%%%%%%%%%%%%%%%%%%%%%%%%%%%%%%%%%%%%%%%%%%%%%%%%%%%%%%%%%%%%

\section{Proofs of Lemmas  \ref{lemma: integer point transform of the half-open cone} and \ref{lemma:integer point transform of a partially-open vertex tangent cone}, Proposition \ref{prop:finite-difference-Bernoulli-Barnes} and Theorem \ref{thm: moments of dilated half-open parallelepipeds}}
\label{sec: proof of theorem moments}

\begin{proof} 
\hypertarget{proof of lemma: integer point transform of the half-open cone} (of Lemma \ref{lemma: integer point transform of the half-open cone})
Let $\bm m\in(\bm v+\mathcal K^{\#})\cap\Z^d$. Because the cone is simplicial, we have: 
\begin{align*}
\bm m=\bm v+\lambda_1\bm w_1+\dots+\lambda_d\bm w_d,
\end{align*}
for some $\lambda_1,\dots,\lambda_d\in\R$. 
Here $\lambda_i>0$ if $i\in I$, and  $\lambda_j\ge0,$ if $j\not\in I$.

The proof essentially follows Theorem 3.5 of \cite{CCD} verbatim, but with a slight difference.
Here we define the following ``twisted" fractional and integer parts, for clarity of 
exposition:
    \begin{align*}
        \{x\}^{\#}&\coloneq\begin{cases}
            \{x\}\text{, if }x\not\in\Z\\
            1\text{, if }x\in \Z
        \end{cases}\\
        \lfloor x\rfloor^{\#}\coloneq x-\{x\}^{\#}&=\begin{cases}
            \lfloor x\rfloor\text{, if }x\not\in\Z\\
            x-1\text{, if }x\in \Z.
        \end{cases}
    \end{align*}
    Thus, we can write
    \begin{align*}
        \bm m&=\bm v+\sum_{j\in I^{c}}(\{\lambda_j\}+\lfloor\lambda_j\rfloor)\bm w_j+\sum_{i\in I}(\{\lambda_i\}^{\#}+\lfloor\lambda_i\rfloor^{\#})\bm w_i\\
        &=\bm v+\sum_{j\in I^{c}}\{\lambda_j\}\bm w_j
        +\sum_{i\in I}\{\lambda_i\}^{\#}\bm w_i
        +\sum_{j\in I^{c}}\lfloor\lambda_j\rfloor \bm w_j
        +\sum_{i\in I}\lfloor\lambda_i\rfloor^{\#} \bm w_i.
    \end{align*}
    Since $\{\lambda_i\}^{\#}\in(0,1],\ i\in I$, the vector 
    \begin{align*}
        \bm p=\bm v+\sum_{j\in I^{c}}\{\lambda_j\}\bm w_j+\sum_{i\in I}\{\lambda_i\}^{\#}\bm w_i
    \end{align*}
    lies in $\bm v+\Pi^{\#}$. Moreover, $\bm p$ is also in $\Z^d$, because it is a difference of the integer vectors $\bm m$ and $\sum_{j\in I^{c}}\lfloor\lambda_j\rfloor \bm w_j
        +\sum_{i\in I}\lfloor\lambda_i\rfloor^{\#} \bm w_i$.  
        Every vector $\bm m\in (\bm v+\mathcal K^{\#})\cap\Z^d$ can be expressed in the form
        \begin{equation}
        \label{eq: generic lattice vector in a half-open cone}
            \bm m=\bm p+k_1\bm w_1+\dots+k_d\bm w_d,
        \end{equation}
        with $\bm p\in(\bm v+\Pi^{\#})\cap\Z^d$ and integers $k_1,\dots,k_d\ge0$. Uniqueness follows from linear independence of the generators and the half-open definition of the  parallelepiped.  
        Consequently, the rational function on the right-hand side of \eqref{eq: integer point transform of the half-open cone} expands as
        \begin{align*}
            \frac{\sigma_{\bm v+\Pi^{\#}}(\bm z)}{(1-e^{\langle \bm w_1,\bm z\rangle})\dots(1-e^{\langle \bm w_d,\bm z\rangle})}
            &=\sum_{\bm p\in(\bm v+\Pi^{\#})\cap\Z^d}e^{\langle\bm p,\bm z\rangle}\sum_{k_1\geq0}e^{k_1\langle \bm w_1,\bm z\rangle}\sum_{k_2\geq0}e^{k_2\langle \bm w_2,\bm z\rangle}\dots\sum_{k_d\geq0}e^{k_d\langle \bm w_d,\bm z\rangle},
        \end{align*}
        and a typical exponent is $\langle\bm p+k_1\bm w_1+\dots+k_d\bm w_d,\bm z\rangle$, which looks exactly like \eqref{eq: generic lattice vector in a half-open cone}.
\end{proof}

\bigskip
\begin{proof}
\hypertarget{proof of lemma: integer point transform of a partially-open vertex tangent cone} (of Lemma \ref{lemma:integer point transform of a partially-open vertex tangent cone})
At each vertex of $\Pi$, the  half-open  parallelepiped is generated by the same \textbf{primitive} generating vectors of $\Pi$, up to a choice of sign. The remaining proof is identical to the proof of Lemma  \ref{lemma: integer point transform of the half-open cone}.
\end{proof}

\begin{proof}
\hypertarget{proof of proposition: prop:finite-difference-Bernoulli-Barnes} (of Proposition \ref{prop:finite-difference-Bernoulli-Barnes})
We recall the definition of the Barnes polynomials via their generating function:
\begin{equation}
    \frac{x^d e^{ux}}
{\prod_{j=1}^d(e^{a_j x}-1)}
=
\sum_{k\geq 0}B_k(u,\va)\frac{x^k}{k!}.
\end{equation}

Substituting $u=ta_I$ gives
\begin{equation}
   \frac{x^d e^{ta_I x}}
{\prod_{j=1}^d(e^{a_j x}-1)}
=
\sum_{k\geq 0}B_k(ta_I,\va)\frac{x^k}{k!}. 
\end{equation}
Multiplying the latter equation by $(-1)^{d-|I|}$ and summing over all
subsets $I\subseteq[d]$, we obtain
\begin{equation}
\label{step before cancellation}
\sum_{I\subseteq[d]}
(-1)^{d-|I|}
\frac{x^d e^{ta_I x}}
{\prod_{j=1}^d(e^{a_j x}-1)}
=
\sum_{k\geq 0}
\left(
\sum_{I\subseteq[d]}
(-1)^{d-|I|}
B_k(ta_I,\va)
\right)
\frac{x^k}{k!}.
\end{equation}
The left-hand side of 
\eqref{step before cancellation} is:

\begin{align}
\frac{x^d}
{\prod_{j=1}^d(e^{a_j x}-1)}
\sum_{I\subseteq[d]}
(-1)^{d-|I|}e^{ta_I x}
&=
\frac{x^d}
{\prod_{j=1}^d(e^{a_j x}-1)}
\prod_{j=1}^d(e^{ta_j x}-1)
= x^dt^d+O(x^{d+1}).
\end{align}
The latter equality follows from
\begin{align*}
    \frac{e^{ta_jx}-1}{e^{a_jx}-1}&=t+O(x)
 \    \implies \ 
 \prod_{j=1}^d\frac{e^{ta_jx}-1}{e^{a_jx}-1}
 =
 t^d+O(x).
\end{align*}
Therefore, \eqref{step before cancellation} gives us the identity
\begin{equation}
\label{formal power series}
x^dt^d+O(x^{d+1}) = \sum_{k\geq 0}
\left(
\sum_{I\subseteq[d]}
(-1)^{d-|I|}
B_k(ta_I,\va)
\right)
\frac{x^k}{k!}.
\end{equation}
Hence the coefficients for $k=0,\dots,d-1$ must vanish:
\begin{equation}
    \sum_{I\subseteq[d]}(-1)^{d-|I|}B_k(ta_I,\va)=0,
\end{equation}
for $k=0,\dots,d-1.$
Moreover, the coefficient of $x^d$ is exactly $t^d$, yielding
\begin{equation}
    \sum_{I\subseteq[d]}(-1)^{d-|I|}B_d(ta_I,\va)=d!t^d.
\end{equation}

\noindent
Finally, for the special case $t=1$ and $k>d$, we use the exact cancellation
\begin{equation}
\label{eq:exact-cancellation-t-equals-one}
\frac{x^d}
{\prod_{j=1}^d(e^{a_jx}-1)}
\prod_{j=1}^d(e^{a_jx}-1)
=
x^d.
\end{equation}
Thus, when $t=1$, equation \eqref{step before cancellation} becomes
\begin{equation}
\label{eq:t-equals-one-generating-function}
\sum_{k\geq 0}
\left(
\sum_{I\subseteq[d]}
(-1)^{d-|I|}
B_k(a_I,\va)
\right)
\frac{x^k}{k!}
=
x^d.
\end{equation}
Therefore every coefficient except the coefficient of $x^d$ vanishes and we get:
\begin{equation}
\label{eq:t-equals-one-coefficients-vanish}
\sum_{I\subseteq[d]}
(-1)^{d-|I|}
B_k(a_I,\va)
=
0,
\qquad
k\neq d.
\end{equation}
In particular, for $k>d$,
\begin{equation}
\label{eq:t-equals-one-high-coefficients-vanish}
\sum_{I\subseteq[d]}
(-1)^{d-|I|}
B_k(a_I,\va)
=
0.
\end{equation}
\end{proof}

\begin{proof}
\hypertarget{proof of theorem: thm: moments of dilated half-open parallelepipeds}
(of Theorem \ref{thm: moments of dilated half-open parallelepipeds})
We begin with the discrete Brion theorem for the half-open parallelepiped:
\begin{equation}
\label{eq: Brion for half-open parallelepiped blue}
\sigma_{t\Pi}(\bm z)
=
\sum_{I\subseteq[d]}
\sigma_{t\bm v_I+\K_{\bm v_I}^{\#}}(\bm z),
\end{equation}
an identity valid for almost all $\bm z\in\C^d$ and all rational $t>0$.

\noindent
{\bf Step $1$}.
Fix a subset $I\subseteq[d]$. By
\eqref{eq: integer point transform of vertex tangent half-open cone}, the
integer point transform of the vertex tangent cone at $t\bm v_I$ is
\begin{equation}
\label{eq: integer point transform of a real dilated half-open cone blue}
\sigma_{t\bm v_I+\K_{\bm v_I}^{\#}}(\bm z)
=
\frac{
\sigma_{t\bm v_I+\Pi_{\bm v_I}^{\#}}(\bm z)
}{
\prod_{j\in I^c}
\left(1-e^{\langle \bm w_j,\bm z\rangle}\right)
\prod_{i\in I}
\left(1-e^{-\langle \bm w_i,\bm z\rangle}\right)
}.
\end{equation}
We now substitute $\bm z=x\bm z_0$, where $\bm z_0$ is generic, and let
$\va :=
\bigl(
\langle \bm w_1,\bm z_0\rangle,
\ldots,
\langle \bm w_d,\bm z_0\rangle
\bigr)^T$. 
Since
\begin{equation}
\label{eq:denominator-sign-change}
\prod_{j\in I^c}
\left(1-e^{x\langle \bm w_j,\bm z_0\rangle}\right)
\prod_{i\in I}
\left(1-e^{-x\langle \bm w_i,\bm z_0\rangle}\right)
=
(-1)^{d-|I|}
e^{-x\langle \bm v_I,\bm z_0\rangle}
\prod_{r=1}^d
\left(
e^{x\langle \bm w_r,\bm z_0\rangle}-1
\right),
\end{equation}
we obtain
\begin{align}
\sigma_{t\bm v_I+\K_{\bm v_I}^{\#}}(x\bm z_0)
&=
(-1)^{d-|I|}
e^{x\langle \bm v_I,\bm z_0\rangle}
\sigma_{t\bm v_I+\Pi_{\bm v_I}^{\#}}(x\bm z_0)
\frac{1}{
\prod_{r=1}^d
\left(
e^{x\langle \bm w_r,\bm z_0\rangle}-1
\right)
}
\nonumber\\
&=
(-1)^{d-|I|}
e^{x\langle(1-t)\bm v_I,\bm z_0\rangle}
\sigma_{t\bm v_I+\Pi_{\bm v_I}^{\#}}(x\bm z_0)
\sum_{k\geq 0}
B_k
\bigl(
t\langle \bm v_I,\bm z_0\rangle,\va
\bigr)
\frac{x^{k-d}}{k!}.
\label{eq:correct-local-cone-expansion}
\end{align}
We note an important subtlety here - we  deliberately kept the translate
$\Pi_{\bm v_I}^{\#}+t\bm v_I$. This point is important because one cannot replace
$\sigma_{\Pi_{\bm v_I}^{\#}+t\bm v_I}(x\bm z_0)$ by an expression involving
$\sigma_{\Pi_{\bm v_I}^{\#}+\bm v_I}(x\bm z_0)$ unless
\begin{equation}
\label{eq:integral-translation-condition-proof}
(t-1)\bm v_I\in\Z^d.
\end{equation}
Indeed, translation by a \emph{nonintegral} vector changes the set of lattice
points, and therefore does not merely multiply the integer point transform
by an exponential factor. The correct local factor is therefore precisely
\begin{equation}
\label{eq:correct-local-factor-proof}
e^{x\langle(1-t)\bm v_I,\bm z_0\rangle}
\sigma_{t\bm v_I+\Pi_{\bm v_I}^{\#}}(x\bm z_0).
\end{equation}

\noindent\textbf{Step $2$.}
We now rewrite this local factor 
\eqref{eq:correct-local-factor-proof}
as a lattice-flow sum over the fixed
half-open parallelepiped $\Pi$. Since
\begin{equation}
\label{eq:vertex-parallelepiped-translate}
\Pi_{\bm v_I}^{\#}
=
\Pi-\bm v_I,
\end{equation}
we have
\begin{equation}
\label{eq:translated-vertex-parallelepiped}
t\bm v_I+\Pi_{\bm v_I}^{\#}
=
\Pi+(t-1)\bm v_I.
\end{equation}
Therefore
\begin{align}
e^{x\langle(1-t)\bm v_I,\bm z_0\rangle}
\sigma_{t\bm v_I+\Pi_{\bm v_I}^{\#}}(x\bm z_0)
&=
\sum_{\bm p\in
(t\bm v_I+\Pi_{\bm v_I}^{\#})\cap\Z^d}
e^{x\langle \bm p-(t-1)\bm v_I,\bm z_0\rangle}
\nonumber\\
&=
\sum_{\bm q\in
\Z^d-(t-1)\bm v_I
\;(\hspace{-.1cm}\bmod \Pi)}
e^{x\langle \bm q,\bm z_0\rangle}.
\label{eq:lattice-flow-local-factor}
\end{align}
Substituting \eqref{eq:lattice-flow-local-factor} into
\eqref{eq:correct-local-cone-expansion}, and then summing over all vertices, we have:
\begin{align}
\sigma_{t\Pi}(x\bm z_0)
&=
\sum_{I\subseteq[d]}
(-1)^{d-|I|}
\sum_{\bm q\in
\Z^d-(t-1)\bm v_I
\  (  \hspace{-.1cm}   \bmod \Pi)}
e^{x\langle \bm q,\bm z_0\rangle}
\sum_{k\geq 0}
B_k
\bigl(
t\langle \bm v_I,\bm z_0\rangle,\va
\bigr)
\frac{x^{k-d}}{k!}
\nonumber\\
&=
\sum_{I\subseteq[d]}
\sum_{\bm q\in
\Z^d-(t-1)\bm v_I
\;(\hspace{-.1cm}\bmod \Pi)}
\sum_{n,k\geq 0}
(-1)^{d-|I|}
\langle \bm q,\bm z_0\rangle^n
B_k
\bigl(
t\langle \bm v_I,\bm z_0\rangle,\va
\bigr)
\frac{x^{n+k-d}}{n!k!}.
\label{eq:final-sum-corrected-proof}
\end{align}

\noindent\textbf{Step $3$.}
 On the other hand,
\begin{equation}
\label{eq:left-side-moment-generating-function}
\sigma_{t\Pi}(x\bm z_0)
=
\sum_{\bm p\in t\Pi\cap\Z^d}
e^{x\langle \bm p,\bm z_0\rangle}
=
\sum_{m\geq 0}
\left(
\sum_{\bm p\in t\Pi\cap\Z^d}
\langle \bm p,\bm z_0\rangle^m
\right)
\frac{x^m}{m!}.
\end{equation}
Comparing the coefficient of $x^m$ in
\eqref{eq:left-side-moment-generating-function} and
\eqref{eq:final-sum-corrected-proof}, with $n+k-d=m$, gives us
\begin{align}
\frac{1}{m!}
\sum_{\bm p\in t\Pi\cap\Z^d}
\langle \bm p,\bm z_0\rangle^m
&=
\sum_{I\subseteq[d]}
\sum_{k=0}^{d+m}
\frac{(-1)^{d-|I|}}{k!(d+m-k)!}
B_k
\bigl(
t\langle \bm v_I,\bm z_0\rangle,\va
\bigr)
\nonumber\\
&\qquad\qquad\qquad\qquad
\times
\sum_{\bm q\in
\Z^d-(t-1)\bm v_I
\;(\hspace{-.1cm}\bmod \Pi)}
\langle \bm q,\bm z_0\rangle^{d+m-k}.
\label{eq:coefficient-comparison-proof}
\end{align}
By the definition of the polytope Dedekind sums, the inner sum is
\begin{equation}
\label{eq:inner-sum-is-dedekind-moment}
\mu_{d+m-k}
\bigl(
\Pi,\bm z_0,(t-1)\bm v_I
\bigr).
\end{equation}
Hence
\begin{align}
\sum_{\bm p\in t\Pi\cap\Z^d}
\langle \bm p,\bm z_0\rangle^m
&=
\frac{m!(-1)^d}{(d+m)!}
\sum_{I\subseteq[d]}
\sum_{k=0}^{d+m}
\binom{d+m}{k}
(-1)^{|I|}
B_k
\bigl(
t\langle \bm v_I,\bm z_0\rangle,\va
\bigr)
\nonumber\\
&\qquad\qquad\qquad\qquad
\times
\mu_{d+m-k}
\bigl(
\Pi,\bm z_0,(t-1)\bm v_I
\bigr).
\label{eq:discrete-moment-of-half-open-parallelepiped-corrected}
\end{align}
We notice that both sides of 
\eqref{eq:discrete-moment-of-half-open-parallelepiped-corrected} may now be easily extended to all positive real $t$, because they are both piecewise analytic functions of $t>0$ which already agree for all rational positive $t$. 
This proves the desired moment formula. Finally, setting $m=0$ gives the
Ehrhart quasi-polynomial formula in part
\eqref{cor:Ehrhart quasi-polynomial of a half-open integer parallelepiped}.
\end{proof}

\section{Further remarks}
\label{sec: further remarks}
The following conjecture stands in sharp contrast to the complete period-collapse phenomena that occurs in the setting of closed rational polytopes.  In the case of integral half-open parallelepipeds $\Pi \subset \R^d$, it follows from Theorem \ref{thm:partial-alternating-differences} part \ref{part c of Ehrhart coefficients}
that their quasi-coefficients all have a period of $1$. This suggests the study of the smallest rational periods smaller than $1$.

\begin{conjecture}[Generic moment-period version]
For \(m\ge 0\), define
\[
\sum_{\bm p\in t\Pi\cap\mathbb Z^d}\langle \bm p,\bm z\rangle^m 
=
\sum_{r=0}^{d+m} c_{m,r}(t;\bm z)t^r, 
\]
the real quasi-polynomial expansion
obtained from the Barnes-polynomial and polytope Dedekind sum formula of
Theorem \ref{thm: moments of dilated half-open parallelepipeds}. Then:

\begin{enumerate}[(a)]
\item Outside a proper algebraic subset of choices of
\(\bm z\in\mathbb C^d\), the  coefficient functions
\(c_{m,r}(t;\bm z)\) have common smallest period equal to \(1\).  In other words, there is no complete period collapse.

\item  What is the smallest rational period of each coefficient $c_{m,r}(t;\bm z)$?
\end{enumerate}
\hfill 
\scalebox{.8}{$\smallpolytope$}
\end{conjecture}

\bigskip

\begin{remark}
The  observation that the left-hand side of \eqref{eq:final-sum-corrected-proof} is analytic yields vanishing identities for the polytope Dedekind sums, arising from the vanishing of the principal part of the Laurent series in the right-hand side of equation \eqref{eq:final-sum-corrected-proof}. These identities may be interpreted as a recursion relation for the polytope Dedekind sums. Indeed, for each $0\leq m \leq d-1$, and all $t>0$, we have the following vanishing identities for the polytope Dedekind sums:
    \begin{equation}
    \label{eq: vanishing identities for Polytope Dedekind Sums}
        0=\sum_{k=0}^m\binom{m}{k}\sum_{I\subseteq[d]}(-1)^{|I|}B_k(t\langle \bm v_I,\bm z\rangle,\va)
        \mu_{m-k}(\Pi,\bm z,(t-1)\bm v_I).
    \end{equation}

Equivalently, for $1\le m \le d-1$, we have
\begin{align}
\label{eq: equivalent statement for recursion of polytope dedekind sums}
&\frac{1}{\langle \bm w_1,\bm z\rangle\cdots\langle \bm w_d,\bm z\rangle} 
\sum_{I\subseteq[d]}
(-1)^{|I|}
\mu_{m}\big(
\Pi,\bm z,(t-1)\bm v_I
\big)  \notag\\
&=
\sum_{k=1}^m\binom{m}{k}\sum_{I\subseteq[d]}(-1)^{|I|+1}
B_k\big(t
\langle \bm v_I,\bm z\rangle,\va
\big)
\, \mu_{m-k}
\big(
\Pi,\bm z,(t-1)\bm v_I
\big).
\end{align}
\hfill 
\scalebox{.8}{$\smallpolytope$}
\end{remark}

\begin{remark}
Unlike \(\bm z_S\) for a proper subset \(S\subsetneq[d]\), the vector
\(\bm z_{[d]}\) is generic, because
\[
\langle \bm w_j,\bm z_{[d]}\rangle=1\notin2\pi i\mathbb Z
\]
for every \(j\).
Moreover, the constant coefficient measures precisely the failure of
the alternating sum in Theorem~1 to continue vanishing at the
critical degree \(k=d\):
\begin{equation}
\sum_{I\subseteq[d]}
(-1)^{|I|}
\mu_d(\Pi,\bm z,t\bm v_I)
=
(-1)^d d!
\left(\prod_{j=1}^{d}\langle \bm w_j,\bm z\rangle\right)c_0(t).
\end{equation}
\hfill
\scalebox{.8}{$\smallpolytope$}
\end{remark}

\begin{remark}
    We note, without proof, that an easy adaptation of the proof of Theorem \ref{thm: moments of dilated half-open parallelepipeds} yields a second derivation of Lanphier's multivariable Faulhaber-type formula \cite{Lanphier}. Indeed, let $\Pi=[0,n_1)\times\dots\times[0,n_d)\subset\R^d$ be a half-open rectangular box. Fix $m\in\Z_{\geq 0}, \, t\in\Z_{>0}$, and $\bm z=(z_1,\dots,z_d)\in\C^d$, with $z_i\neq0, \ i=1,\dots,d$.

    For the rectangular box $\Pi$, every vertex half-open parallelepiped is a half-open unit cube, whose polytope Dedekind sum is trivial. 
 After the smoke clears,  we are left with the following alternating sum:
\begin{equation}
\label{eq: discrete moment of smooth parallelepiped}
\sum_{\bm p\in t\Pi\cap\Z^d}
\langle \bm p,\bm z\rangle^m
=
\frac{m!}{(d+m)!}
\sum_{I\subseteq[d]}
(-1)^{d-|I|}
B_{d+m}
\bigl(t\langle \bm n_I,\bm z\rangle,\bm z\bigr),
\end{equation}
where $\bm n_I\coloneq \sum_{i\in I}n_i\bm e_i$.
For the sake of completeness, we also present a new  elementary proof of identity
\eqref{eq: discrete moment of smooth parallelepiped}:
\begin{align}
\sum_{j\ge d}
\left(
\sum_{\bm k\in t\Pi\cap\Z^d}
\langle \bm k,\bm z\rangle^{j-d}
\right)
\frac{x^j}{(j-d)!}
&=
x^d
\sum_{\bm k\in t\Pi\cap\Z^d}
e^{x\langle \bm k,\bm z\rangle} \nonumber\\
&=
x^d
\prod_{i=1}^d
\left(
\sum_{k_i=0}^{tn_i-1} e^{xk_i z_i}
\right) \nonumber\\
&=
x^d
\prod_{i=1}^d
\frac{e^{xtn_i z_i}-1}{e^{xz_i}-1} \nonumber\\
&=
\sum_{I\subseteq[d]}
(-1)^{d-|I|}
\frac{x^d e^{xt\langle\bm n_I,\bm z\rangle}}
{\prod_{i=1}^d(e^{xz_i}-1)} \nonumber\\
&=
\sum_{l\ge 0}
\sum_{I\subseteq[d]}
(-1)^{d-|I|}
B_l(t\langle\bm n_I,\bm z\rangle,\bm z)
\frac{x^l}{l!}.
\label{eq: Bernoulli series expansion for nth discrete moment}
\end{align}
Equating the coefficient of $x^{n+d}$ gives identity \eqref{eq: discrete moment of smooth parallelepiped}.
    \hfill\scalebox{.8}{$\smallpolytope$}
\end{remark}

\bigskip
\noindent
{\bf AI Usage}

\medskip
\noindent
The authors used ChatGPT for editorial assistance, proofreading, and suggestions during the writing of this paper.  The authors independently provided and re-checked all
mathematical arguments, statements, and references, and take full responsibility for the final manuscript.
 
%%%%%%%%%%%%%%%%%%%%%%%

\bibliographystyle{alpha}
\bibliography{Reference}

@article{BeckSamWoods,
    author={Beck, Matthias and Sam, Steven and Woods, Kevin},
    title={Maximal periods of ({Ehrhart}) quasi-polynomials},
    journal={Journal of Combinatorial Theory, Series A},
    volume = {115},
    number = {3},
    pages = {517--525},
    year = {2008},
}

@article{Brion,
    author={Brion, Michel},
    title={Points entiers dans les polyèdres convexes},
    journal={Annales scientifiques de l’É.N.S},
    volume = {21},
    number = {4},
    pages = {653--663},
    year = {1988},
}

@article{Ehrhart,
    author={Ehrhart, Eugène},
    title={Sur les polyèdres rationnels homothétiques à n dimensions},
    journal={C. R. Acad. Sci. Paris},
    volume = {254},
    pages = {616--618},
    year = {1962}
}

@article{GunnelsSczech,
    title = {Evaluation of {Dedekind} sums, {Eisenstein} cocycles and special values of {L}-functions},
    journal = {Duke Math. J.},
    volume = {118},
    pages = {229--260},
    year = {2003},
    number={2},
    author = {Gunnells, Paul and Sczech, Robert}
}

@article{Lanphier,
    title = {Multivariable {Faulhaber}-type formulas and {Barnes} zeta functions},
    journal = {Journal of Mathematical Analysis and Applications},
    volume = {506},
    number = {2},
    pages = {1--13},
    year = {2022},
    author = {Lanphier, Dominic}
}

@article{Stanley1974,
    author ={Stanley, Richard P.},
    title ={Combinatorial reciprocity theorems},
    journal ={Advances in Mathematics},
    year = {1974},
    pages={194--253},
    volume={14},
    number={2}
}

@book{Barvinok,
  title={Integer Points in Polyhedra},
  author={Barvinok, Alexander},
  year={2008},
  publisher={European Mathematical Society},
  series={Zurich Lectures in Advanced Mathematics}
}

@book{CCD,
  title={Computing the continuous discretely},
  author={Beck, Matthias and Robins, Sinai},
  year={2015},
  publisher={Springer},
  series={Undergraduate Texts in Mathematics},
  edition={Second}
}

@book{Robinsbook,
  title={Fourier Analysis on Polytopes and the Geometry of Numbers, Part I: A friendly introduction},
  author={Robins, Sinai},
  year={2024},
  publisher={American Mathematical Society},
  series={Student Mathematical Library}
}

@book{Stanley,
  title={Enumerative combinatorics},
  edition={Second},
  author={Stanley, Richard P.},
  publisher={Cambridge University Press},
  year={2011}
}

@misc{Robins2026,
  title={Ehrhart quasi-polynomials via {Barnes} polynomials and
  discrete moments of parallelepipeds (preprint)},
  author={Robins, Sinai},
  year={2026},
  eprint={2601.12596},
  archivePrefix={arXiv},
  primaryClass={math.CO},
  url={https://arxiv.org/abs/2601.12596}
}

@article{KidwaiOsuga2025,
  author  = {Omar Kidwai and Kento Osuga},
  title   = {Refined BPS Structures and Topological Recursion---the Weber and Whittaker Curves},
  journal = {International Mathematics Research Notices},
  year    = {2025},
  doi     = {10.1093/imrn/rnaf116}
}
\section*{Authors' addresses}

\noindent
\textbf{Sinai Robins}\\
Universidade de São Paulo, Instituto de Matemática e Estatística,
São Paulo, SP, Brazil\\
\texttt{sinai.robins@gmail.com}

\vspace{1em}

\noindent
\textbf{André Rosenbaum Coelho}\\
Universidade de São Paulo, Instituto de Matemática e Estatística,
São Paulo, SP, Brazil\\
\texttt{andrerosenbaumcoelho@usp.br}
\end{document}